\documentclass[12pt,a4paper,article]{memoir} 
\usepackage[top=2in, bottom=1.5in, left=1.5in, right=1.5in]{geometry}                
\usepackage{graphicx}
\usepackage{amssymb}
\usepackage{epstopdf}
\usepackage{appendix}
\usepackage{booktabs} 
\usepackage{array} 
\usepackage{paralist} 
\usepackage{verbatim} 
\usepackage{subfig} 
\usepackage[utf8]{inputenc} 
\usepackage{amsmath}
\usepackage{amsthm}

\usepackage{mathtools}
\usepackage[all]{xy}
\usepackage{rotating}
\usepackage{amssymb,amsbsy,amsthm,amsmath,graphicx,epsfig}

\usepackage{amscd}
\usepackage{amsfonts}
\usepackage{mathrsfs}
\usepackage{setspace}
\usepackage{version}

\theoremstyle{plain}
\numberwithin{equation}{section} \theoremstyle{plain}
\newtheorem{theorem}[section]{Theorem}
\newtheorem{proposition}[section]{Proposition}
\newtheorem{lemma}[section]{Lemma}

\newtheorem{definition}[section]{Definition}
\newtheorem{conjecture}[section]{Conjecture}

\newtheorem{claim}[section]{Claim}
\newtheorem{example}[section]{Example}

\newtheorem{remark}[section]{Remark}

\renewcommand{\leq}{\leqslant}
\renewcommand{\geq}{\geqslant}
\newsavebox{\proofbox}
\savebox{\proofbox}{\begin{picture}(7,7)  \put(0,0){\framebox(7,7){}}\end{picture}}

\title{Exposition of Elekes-Szab\'{o} paper}
\author{Wang Hong}
\date{} 

\begin{document}

\maketitle
This is a memoire of stage M2 of l'Universit\'{e} de Paris-sud. The author did her master memoire at UCLA (IPAM) under the mentorship of Joszef Solymosi and Emmanuel Breuillard.
\tableofcontents* 

\chapter{Introduction}
Given an $n\times n\times n$ Cartesian product in $\mathbb{C}^{3}$, and a complex algebraic surface $V$. Elekes and Szab\'{o} proved \cite{elekes-szabo} that if $V$ has large intersection with this Cartesian product, then it has a multiplicative form. This paper gives a quantitative refinement of the theorem of Elekes and Szabo\'{o}.

 Various mathematical problems can be transformed to this type of question: intersection between an algebraic variety and a grid. For example, Sharir, Sheffer and Solymosi's result \cite{sharir2013distinct} on distinct distances problem on two lines, Pach and de Zeeuw's result \cite{pach2013distinct} on distinct distances on algebraic curves in the plane, as well as triple points between three families of unit circles problem \cite{elekes1999sums} \cite{raz2014triple}. 
 
Elekes and R\'{o}nyai first established a result \cite{elekes2000combinatorial} on the case where $V$ is the graph of a polynomial (rational function) in $\mathbb{C}[x,y]$, and Raz-Sharir-Solymosi\cite{raz2014polynomials} improved the bound in \cite{elekes2000combinatorial} to $n^{11/6}$. However, for most applications, one could not express an algebraic relation as an explicit function.  So one needs a theorem on algebraic relations of implicit function.  Elekes and Szab\'{o} \cite{elekes-szabo} proved this result for an arbitrary two-dimensional algebraic surface $V$ with bound $n^{2-\eta}$ for implicit small $\eta$, and generalized it to certain dimension $2k$ algebraic variety on $(k\times k\times k)$-dimensional space. In addition, Emmanuel Breuillard  showed that if $\eta$ is small enough, the algebraic group identified is in fact nilpotent (see future preprint of E. Breuillard and H. Wang).
We improved this bound to $n^{11/6+\epsilon}$ for arbitrary small $\epsilon>0$ and give an explicit bound for the latter case. In the real case (where $A$, $B$ and $C$ are real lines $\mathbb{R}$), we obtained the bound without $\epsilon$; this result is independently obtained by Raz, Sharir and de Zeeuw. Another variation is a result of Tao \cite{tao2012expanding} on expanding polynomials on large characteristic finite fields. 

The main goal of our paper is to explain the original Elekes-Szab\'{o} paper in more details and establish an improved bound on the exponent of $n$. For impatient readers who want to know Elekes-Szab\'{o} theorem, we give a proof of it in Chapter~\ref{section: elekes-szabo} assuming the main technical lemma, which is proved in Chapter~\ref{section: composition lemma}. And Chapter~\ref{section: complex} includes an improved incidence theorem on $\mathbb{C}^{3}$, which results in the improved $11/6+\epsilon$ bound on Elekes-Szab\'{o} Theorem. In their original paper, Elekes and Szab\'{o} established a theorem on higher dimension, replacing $\mathbb{C} \times \mathbb{C} \times \mathbb{C} $ with higher dimensional varieties $A\times B\times C$ where $\dim(A)=\dim(B)=\dim(C)$. We also includes its proof in Chapter ~\ref{section: elekes-szabo} with explicit exponents on $n$, and the corresponding incidence theorem is established in Chapter~\ref{section: incidence theorem}. 

It is worth noting that the technical part, Composition Lemma, of Elekes-Szab\'{o} theorem relies heavily on algebraic geometry (or model theory \cite{hrushovski1986contributions}).  And Guth-Katz \cite{guth2010erdos} applied polynomial methods to solve joints and distinct distances problems. This coincidence suggests the surprising application of algebraic method in combinatorics problems. 

\section*{Notation.} We use the usual asymptotic notation $X=\mathcal{O}(Y)$ or $X\lesssim Y$ to denote the estimate $X\leq CY$ for some absolute constant $C$. If we need the implied constant $C$ to depend on additional parameters, we indicate this by subscripts, thus for instance $X=\mathcal{O}_{d}(Y)$ or $X\lesssim_{d} Y$ for some quantity $C_{d}$ depending on $d$.

\section*{Acknowledgements}
The author is  very grateful to Emmanuel Breuillard and Jozsef Solymosi for comments and references, and Yang Cao for helpful discussions on algebraic geometry. This research was performed while the author was visiting the Institute for Pure and Applied Mathematics (IPAM), which is supported by the National Science Foundation.

\chapter{Preliminaries}
There are two types of combinatorics problems we are interested in: point-variety incidence problems and sum-product problems.  Among which, we have three main conjectures: Erd\"{o}s distinct distances conjecture, Erd\"{o}s unit distances conjecture and sum-product conjecture. 
\begin{conjecture}
(Erd\"{o}s distinct distances Conjecture) Let $N$ be a large natural number, the least number $\#\{|x_{i}-x_{j}|: 1\leq i <j\leq N\}$ of distances that are determined by $N$ points $x_{1},..., x_{N}$ in the plane is $\mathcal{O}(N/\sqrt{N})$.
\end{conjecture}

\begin{conjecture}(Erd\"{o}s unit distances Conjecture) Let $N$ be a large natural number, the largest number $\#\{(x_{i}, x_{j}): 1\leq i <j \leq N, |x_{i}-x_{j}|=1\}$ of pairs of points in $x_{1},...,x_{N}$ in the plane that have unit distance is $\mathcal{O}(N^{1+\epsilon})$ for arbitrary small positive $\epsilon$.
\end{conjecture}

\begin{conjecture}(Sum-product Conjecture)
Let $A$ be a finite subset of real numbers. The sumset and productset of $A$ are defined by $ A+A=\{a+b:a,b\in A\}$ and $A\cdot A=\{ab: a,b\in A\} $. Then
\[\max(|A+A|, |A\cdot A|)\geq \mathcal{O}(|A|^{2-\epsilon})\] for arbitrary small positive $\epsilon$.
\end{conjecture}
The three conjectures, as simple as their statements are,  have many applications to other areas: harmonic analysis, number theory, computer science, etc. In this chapter, we will explain very briefly how incidence geometry and sum-product theory relate to the three conjectures and how they are connected to the subject of this memoire: Elekes-Szab\'{o} Theorem.
\subsection*{Incidence geometry}

Given a finite collection $P$ of points in $\mathbb{F}^{n}$ ( $\mathbb{F} = \mathbb{R}, \mathbb{C}$ or finite fields), and a finite collection $L$ of subvarieties of $\mathbb{F}^{n}$, incidence geometry studies the cardinality of 
\[\mathcal{I}(P,L):=\{(p,l)\in P\times L: p\in l\}\]
First let us state a theorem of Szemer\'{e}di and Trotter on point-line incidences over plane:

\begin{theorem}(Szemer\'{e}di-Trotter Theorem). \cite{szemeredi1983extremal}
Let $P$ be a finite set of points in $\mathbb{R}^{2}$, and let $L$ be a finite set of lines in $\mathbb{R}^{2}$. Let $\mathcal{I}(P,L):=\{(p,l)\in P\times L: p\in l\}$ be the set of incidences. Then
\[ |\mathcal{I}(P,L)|\leq C(|P|^{2/3}|L|^{2/3}+|P|+|L|\]
for some absolute constant $C$.
\end{theorem}

This result is optimal. Consider for any positive integer $N\in \mathbb{Z}^{+}$ the set of points on the grid: 
\[P=\{(a,b)\in \mathbb{Z}^{2}: 1\leq a\leq N; 1\leq b\leq 2N^{2}\} \] and the set of lines
\[L=\{(x,mx+b): m,b\in \mathbb{Z}; 1\leq m\leq N; 1\leq b\leq N^{2}\}.\]
Clearly, $ |P|=2N^{3}$ and $|L|=N^{3}$. Since each lines is incident to $N$ points, the number of incidences is $N^{4}$ which matches the above bound.

Since then, the Szemer\'{e}di-Trotter theorem has various consequences and applications. And people tend to generalize it: 
incidences between points and pseudo-lines/curves in two-dimensional Euclidean space ~\cite{pach1998number} \cite{wang2013bounds};
incidences between points and algebraic varieties in higher dimension $\mathbb{R}^{d}$ or $\mathbb{C}^{d}$~\cite{solymosi2012incidence}; more difficultly, incidence problem over finite fields; etc. 

 In 2010, Guth and Katz~\cite{guth2010erdos} nearly solved the Erd\"{o}s distinct distance conjecture with the lower bound $\mathcal{O}(N/\log(N))$ by applying an incidence theorem between points and lines in $\mathbb{R}^{3}$. The Erd\"{o}s unit distances problem is closely related to incidences between points and unit circles, which is still little understood. And the best upper bound is given by Spencer, Joel and Szemer\'{e}di via a Szemer\'{e}di-Trotter type bound of incidence problem ~\cite{spencer1984unit}.  Also, incidence theorem is a key step in the proof of Elekes-Szab\'{o} Theorem. 
 
On the other hand, the inverse problem of incidence theorem is very interesting and difficult: if a configuration attain the optimal bound, what can we say about the configuration, does there exist a classification of extremal configurations? The inverse problem is very little understood; even the inverse problem of Szemer\'{e}di-Trotter Theorem, we could not say much about.
\subsection*{Sum-product problem}
It is easy to select a finite set $A$ from real numbers, such that it has small sums (an arithmetic progression for example) or finite subset of small products( by taking a geometric progression). However, one has a simple observation: it is not likely that both sumset and productset are small simultaneously. 

In 1983, Erd\"{o}s and Szemer\'{e}di proved 
\begin{theorem}(Erd\"{o}s-Szemer\'{e}di Theorem~\cite{erdHos1983sums})
Let $A$ be a finite subset of real numbers. There exists a positive constant $\epsilon$ such that 
\[\max(|A+A|, |A\cdot A|)\geq \mathcal{O}(|A|^{1+\epsilon}).\]
\end{theorem}
The best known result is $\mathcal{O}(|A|^{4/3})$ by Solymosi in 2009 ~\cite{solymosi2009bounding}.

There have been a lot of structure theorems relating to sum-product problem. For instance, Freiman's Theorem states that if $A$ has small sumset (of linear size), then it is contained in a generalized arithmetic progression.

\begin{definition}
Let $d$ and $n_{1},...n_{d}$ be positive integers and $\Delta_{1},...,\Delta_{d}$ arbitrary real or complex numbers. A set $G$ is a \emph{generalized arithmetic progression}( ``arithmetic GP" for short) of dimension $d$ and size $n=n_{1}\cdot n_{2}\cdot ...\cdot n_{d}$ if 
\[G=\{\sum_{i=1}^{d} k_{i}\cdot\Delta_{i}; 0\leq k_{i} <n_{i} \text{ for } i=1,...,d\},\]
and these elements are all distinct.
\end{definition}

In what follows, $\mathcal{G}^{d,n}$ will denote the class of arithmetic GP's of dimension \emph{not exceeding} $d$ and size \emph{at most} $n$.

\begin{theorem}(Freiman's Theorem ~\cite[Freiman]{freaeiman1973foundations} ~\cite[Ruzsa]{ruzsa1992arithmetical}~\cite{ruzsa1994generalized})
If $|X|, |Y|\geq n$ and $|X+Y|\leq Cn$ then $X\cup Y$ is contained in an arithmetic GP $G\in \mathcal{G}^{d^{*}, C^{**}n}$, where $d^{*}=d^{*}(C)$ and $C^{*}=C^{*}(C)$ do not depend on $n$.
\end{theorem}

\subsection*{Relation with Elekes-Szab\'{o} Theorem}
Before considering an algebraic variety containing many points of a Cartesian product, Elekes first  studied a set of straight lines with small composition sets.  

We denote by $\mathcal{L}$ the set of non-constant real or complex linear functions $x \mapsto ax+b (a\neq 0)$. For two subset $\Phi, \Psi$ of $\mathcal{L}$, the composition set is defined by $\Phi\circ\Psi :=\{\phi\circ\psi; \phi\in\Phi, \psi\in\Psi\}$.

\begin{theorem} (Linear theorem for composition sets ~\cite{elekes1999sums})\label{thm: linear theorem for composition}
For all $c, C>0$ there exists a $c^{*}=c^{*}(c,C)>0$ with the following property.

Let $\Phi, \Psi \subset \mathcal{L}$ and $E\subset \Phi\times\Psi$ with $|\Phi|, |\Psi|\leq N$ and $|E|\geq cN^{2}$. Assume, moreover, that 
\[|\Phi\circ_{E}\Psi|\leq CN.\]
Then there are $\Phi^{*}\subset \Phi$ and $\Psi^{*}\subset\Psi$ for which $|(\Phi^{*}\times\Psi^{*})\cap |E|\geq c^{*}N^{2}$ and 

(i) either both $\Phi^{*}$ and $\Psi^{*}$ consist of functions whose graphs are all parallel (but the directions may be different for $\Phi^{*}$ and $\Psi^{*}$);

(ii) or both $\Phi^{*}$ and $\Psi^{*}$ consist of functions whose graphs all pass through a common point (which may be different for those in $\Phi^{*}$ and in $\Psi^{*}$).
\end{theorem}

Apply Theorem~\ref{thm: linear theorem for composition} Elekes proved a generalization of Freiman's Theorem.

Let $G\subset \mathbb{C}$ be an arithmetic GP, $\Phi,\Psi\subset\mathcal{L}$ and $C$ a positive integer. We say that the pair $(\Phi,\Psi)$ is an arithmetic GP-type structure based upon $G$ with $C$ slopes if there are non-zero complex numbers $s_{1}, s_{2},...,s_{C}$ such that 
\[\Phi^{-1}\cup\Psi=\{x\mapsto s_{i}x+g; 1\leq i\leq C \text{ and } g\in G\}.\]
\begin{theorem}(Elekes\cite{elekes1998linear}) For every $C>0$ there are $C^{*}=C^{*}(C)>0$, $C^{**}=C^{**}(C)>0$ and $d^{*}=d^{*}(C)>0$ with the following property. If $\Phi, \Psi \subset \mathcal{L}$ with $|\Phi|, |\Psi|\geq N$ and 
\[|\Phi\circ\Psi|\leq CN\]
then $(\Phi, \Psi)$ is contained in an arithmetic or in a geometric $GP$-type structure with $\leq C^{*}$ slopes or bunches, respectively, based upon an arithmetic or geometric $G\in \mathcal{G}^{d^{*}, C^{**}N}$.
\end{theorem}
This theorem is a generalization of Freiman's Theorem by taking $\Phi=\Psi=\{x+a_{i}; a_{i}\in X\}$ .

Another fruit of Theorem~\ref{thm: linear theorem for composition} is the Elekes-R\'{o}nyai Theorem. 
\begin{theorem}(Elekes-R\'{o}nyai Theorem~\cite{elekes2000combinatorial})
For every $C\geq 1$ and positive integer $d$,  there is an $n_{0}=n_{0}(C,d)$ with the following property.
If $F\in \mathbb{R}(x,y)$ of degree $d$, and there are $X, Y\in \mathbb{R}$ with $|X|=|Y|=n$ such that 
\[|F(X\times Y)|\leq Cn.\]
Then there are rational functions $f, g, h\in \mathbb{R}(z)$ for which one of (1)-(3) below is satisfied:
\begin{enumerate}
	\item $F(x,y)=f(g(x)+h(y))$;
	\item $F(x,y)=f(g(x)\cdot h(y))$;
	\item $F(x,y)=f\big(\frac{g(x)+h(y)}{1-g(x)\cdot h(y)}\big)$.
\end{enumerate}
Moreover, if $F\in \mathbb{R}[x,y]$ is a polynomial then one of the first two possibilities (i) or (ii) holds with some polynomials $f, g, h \in \mathbb{R}[z]$.
\end{theorem}

\chapter{Reproof of the Elekes-Szab\'{o} theorem}\label{section: elekes-szabo}
In this Chapter, we reprove the Elekes-Szab\'{o} theorem following their original proof focusing on the quantitative improvements.
\begin{theorem}(Elekes-Szab\'{o} Theorem~\cite{elekes-szabo})\label{thm: elekes-szabo}
For any small $\epsilon >0$, and any positive integer $d$ there exist positive constants $\lambda =\lambda (d)$ and  $n_{0} = n_{0}(d, \epsilon)$ with the following property.

If $V \subset  \mathbb{C}^{3}$ is a two-dimensional irreducible algebraic surface of degree $d$ then the following are equivalent:

(a) For at least one $n>n_{0}$ there exist $X, Y, Z \subset \mathbb{C}$ such that $|X|=|Y|=|Z|=n$ and 
\[V\cap (X\times Y\times Z)| \geq n^{2-\eta}, \ \ \eta=\tfrac{1}{6}-\epsilon ;\]

(b) $V$ is either a cylinder over a curve $F(x,y)=0$ or $F(x,z)=0$ or $F(y,z)=0$ or, otherwise, there exist a one-dimensional connected algebraic group $\mathcal{G}$ and analytic multi-functions $f, g, h : \mathcal{G} \rightarrow \mathbb{C}$ of complexity bounded by $\lambda(d)$, such that their inverses are also analytic multi-functions of complexity bounded by $\lambda(d)$, and $V$ is the closure of a component of the $f\times g\times h$-image of the special subvariety: \[\mathcal{G}_{sp}: =\{(x,y,z); x\oplus y\oplus z =0 \in \mathcal{G}\},\]

(c) Let $D\subset \mathbb{C}$ denote the open unit disc. Then either $V$ contains a cylinder over a curve $F(x,y)=0$ or $F(x,z)=0$ or $F(y,z)=0$ or, otherwise, there are one-to-one analytic functions $f, g, h : D\rightarrow \mathbb{C}$ with analytic inverses such that : 
\[V\supset \big\{ \big( f(x), g(y), h(z)\big) \in \mathbb{C}^{3}; x, y, z \in D, x+y+z=0\big\} .\]
\end{theorem}

\begin{remark}The constant $\eta=\tfrac{1}{6}-\epsilon$ comes from Theorem ~\ref{thm: complex} , and we can get $\tfrac{1}{22}-\epsilon$ directly from Elekes-Szab\'{o} paper.  
If $V \subset \mathbb{R}^{3}$, the constant  $\eta=\tfrac{1}{6}-\epsilon$ is derived from Elekes-Szab\'{o} paper, and we obtain $\eta=\tfrac{1}{6}$ with slightly altering the proof.
\end{remark}
We will establish a stronger Theorem ~\ref{thm: main theorem} which implies $(a)$ to $(b)$, for the whole proof of equivalence we refer to Elekes-Szab\'{o} paper \cite[Section 4.2.]{elekes-szabo}. 
\begin{definition}
A \emph{multi-function} $(F: A\rightarrow B)$ between two irreducible projective varieties $A$ and $B$ is a nonempty closed algebraic subset $F\subset A\times B$ such that the two projections of $F$ are generically finite and surjective. Then $A, B$ and $F$ necessarily have equal dimensions. If the projections $\pi_{A}: F\rightarrow A$ and $\pi_{B}: F\rightarrow B$ have degrees $\alpha$ and $\beta$, respectively, then we say that the degree of the multi-function $F$ is $max(\alpha, \beta)$. 
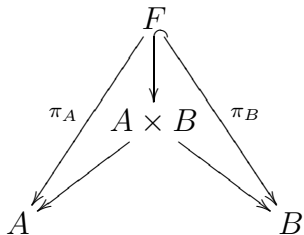
\begin{displaymath}
	\xymatrix{& F \ar@{^{(}->}[d] \ar[ddl]_{\pi_{A}}\ar[ddr]^{\pi_{B}}&\\ & A\times B\ar[dl]\ar[dr]&\\
			A& &B}
\end{displaymath}

\end{definition}

\begin{definition}
A \emph{generalized multi-function} $(F: A \rightarrow B)$ is an algebraic set $F$ together with morphisms $F \rightarrow A$ and $F \rightarrow B$ such that the closure of the image of $F$ in $A \times B$ is a multi-function.  It is called the multi-function represented by $F$.
\end{definition}

Given $(F: A \rightarrow B)$ and $(G: B \rightarrow C)$ two generalized multi-functions, then the fiber product $H=F \times_{B} G$ has natural projections to $A$ and $C$, and it is easy to see that $(H: A\rightarrow C)$ is a generalized multi-function representing the composition of $F$ and $G$.
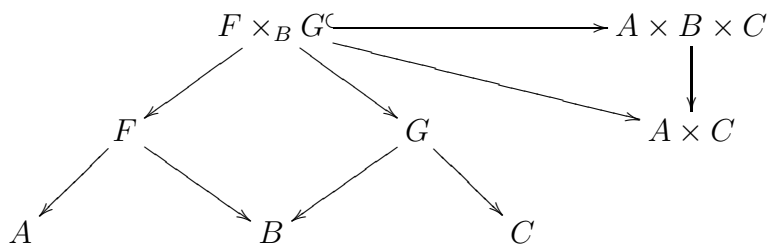
\begin{displaymath}
	\xymatrix{& &F \times_{B} G \ar@{^{(}->}[rrr]\ar[dl]\ar[dr]\ar[drrr] & & & A\times B\times C\ar[d] \\
			&F\ar[dl]\ar[dr] &&G\ar[dl]\ar[dr] &&A\times C\\
			A&&B&&C}
\end{displaymath}
\begin{definition}
A \emph{family of multi-functions} $(F_{t} : A\rightarrow B, t\in T)$ parametrized by the irreducible variety $T$ is a closed algebraic subset $F\subset A\times B\times T$, such that the generic fiber $F_{t}\subset A\times B$ is a multi-function. 
\end{definition}

Here we only work with irreducible parameter spaces. And we can represent $T$ with its reduced scheme structure, hence $T$ is now an integral scheme (irreducible and reduced). We shall study the family of multi-function $F$ between $A$ and $B$ via $A\times B$'s Hilbert scheme $\mathcal{H}(A\times B)$. One can treat $\mathcal{H}(A\times B)$ as a collection of closed subvarieties of $A\times B$, and $\mathcal{H}(A\times B)$ itself has a scheme (algebraic variety) structure, such that every point of $\mathcal{H}(A\times B)$ represents a closed subvariety of $A\times B$. Here we only deal with bounded degree multi-functions, so we could restrict to Hilbert scheme of bounded degree. In this case $\mathcal{H}(A\times B)$ can be viewed as a finite dimensional variety. Moreover, Hilbert scheme $\mathcal{H}(A\times B)$ has the universal property: every family of multi-functions $F$ parametrized by an integral scheme $T$ induces a unique rational map $\phi_{F}$ from $T$ to $\mathcal{H}(A\times B)$ with the property that $\phi_{F}$ is defined on a dense open set $T'$ of $T$, it sends every $t\in T'$ to the point in $\mathcal{H}(A\times B)$ that represents $F_{t}$.  And $\phi_{F}$ does not depend on the choice of $T'$,  in general, one cannot extend it continuously to the whole $T$.
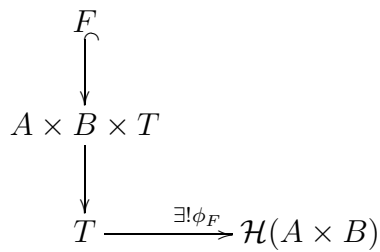
\begin{displaymath}
	\xymatrix{ F \ar@{^{(}->}[d]&\\
A\times B\times T \ar[d]&\\
T\ar[r]^{\exists ! \phi_{F}}& \mathcal{H}(A\times B)}
\end{displaymath}
\begin{definition}
We say that $F$ is a \emph{$k$-dimensional family} if $\phi_{F}(T)$ is $k$-dimensional. Moreover, $F$ is called \emph{equivalent to} another family $(\hat{F}_{u}: A\rightarrow B, u\in U)$ if the rational images of $T$ and $S$ in the Hilbert scheme have the same closure, i.e. if they parametrize essentially the same set of multi-functions.
\end{definition}

Let $R$ denote the closure of $\phi_{F}(T)$ in $\mathcal{H}(A\times B)$, and let $H$ denote the corresponding family of multi
-functions parameterized by $R$. Clearly $H$ is a family of multi-functions that is equivalent to $F$ and that parametrizes each multi-function only once. Moreover, if $\dim T=\dim R$, then $\phi_{F}$ is a multi-function.
\begin{definition}\label{def: common component}
 We say that $F$ and another family $(G_{s}: A\rightarrow B, s\in S) $ have a common component if there are families $(\hat{F}_{u} :A\rightarrow B, u\in U)$ and $(\hat{G}_{u}: A\rightarrow B, u\in U)$ equivalent to them, parametrized by the same $U$, such that for all $u \in U$ the algebraic subvarieties $\hat{F}_{u}$ and $\hat{G}_{u}$ of $A\times B$ have a common component.
\end{definition}
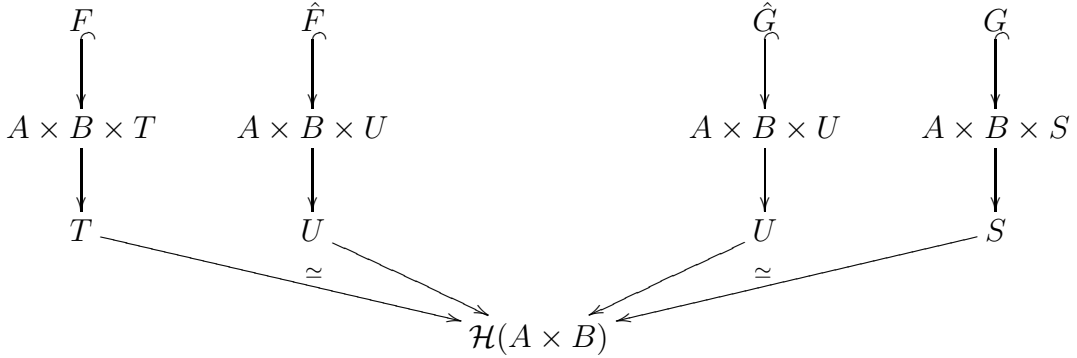
\begin{displaymath}
	\xymatrix{F\ar@{^{(}->}[d]&\hat{F}\ar@{^{(}->}[d]&&\hat{G}\ar@{^{(}->}[d]&G\ar@{^{(}->}[d]\\
		A\times B\times T \ar[d]& A\times B\times U \ar[d]&& A\times B\times U \ar[d] &A\times B\times S\ar[d]\\
		T\ar[drr]^{\simeq}& U\ar[dr] & &U\ar[dl]& S\ar[dll]_{\simeq}\\
		&&\mathcal{H}(A\times B)}
\end{displaymath}
One may ask what's the relation between $F$ and $G$ as a algebraic variety. We will analyze it in the range of Main Theorem. 
\begin{definition}
Let $\Gamma$ be a connected algebraic group acting on a variety $V$. Then the standard family of multi-functions corresponding to this group action is the family $(F_{\gamma}: V \rightarrow V, \gamma \in \Gamma)$ where $F_{\gamma}$ is the graph of the automorphism $\gamma$.  We say that a family $(G_{s}:A\rightarrow B, s\in S)$ is \emph{related to} the standard family $F$ along the multi-functions $(\alpha: V\rightarrow A)$ and $(\beta : V\rightarrow B)$ if $G$ has a common component with the family $(\beta\circ F_{\gamma} \circ \alpha^{-1} : A\rightarrow B, \gamma \in \Gamma)$.
\end{definition}
The multi-function family $(\beta\circ F\circ\alpha^{-1} : A\rightarrow B, \gamma \in \Gamma)$ can be viewed as multi-function image of $(\alpha \times \beta \times Id_{\Gamma}): V\times V\times \Gamma \rightarrow A\times B\times \Gamma$, we name it $H$. Then we go back to the Definition ~\ref{def: common component} that two families of multi-functions $G$ and $H$ have a common component.

Next technical lemma implies where the ``group" comes from. One may use it as a black box, or curious readers might go to Chapter ~\ref{section: composition lemma} for its proof.

\begin{lemma}{$\mathbf{(Composition}$ $\mathbf{ Lemma)}$}
Suppose there are two $k$-dimensional families of multi-functions, $(F_{t}: A \rightarrow B, t\in T)$ and $(G_{s} :B\rightarrow C, s\in S)$, such that the family of compositions $(G_{s}\circ F_{t}: A \rightarrow C, (t,s)\in T\times S)$ has a common component with a $k$-dimensional family. Then there is a $k$-dimensional connected algebraic group $\Gamma$ acting on a variety $V$, and multi-functions $(\alpha: V\rightarrow A), (\beta: V\rightarrow B)$ and $(\gamma: V\rightarrow C)$ such that the family $F$ is related to the standard family corresponding to the $\Gamma$-action on $V$ along $\alpha$ and $\beta$, and the family $G$ is related to it along $\beta$ and $\gamma$. Moreover, the degrees of $\alpha$, $\beta$ and $\gamma$ can be bounded in terms of the degrees of the generic members $F_{t}$ and $G_{s}$.

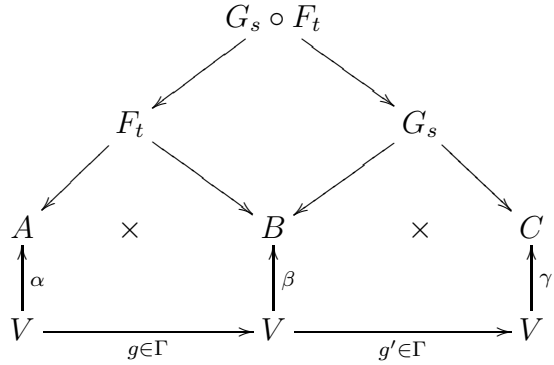
\begin{displaymath}
	\xymatrix{ &&G_{s}\circ  F_{t} \ar[dl] \ar[dr]\\
		&F_{t}\ar[dl] \ar[dr] && G_{s}\ar[dl]\ar[dr]\\
		A &\times &B&\times &C\\
		V\ar[u]_{\alpha}\ar[rr]_{g\in\Gamma} & &V  \ar[u]_{\beta}\ar[rr]_{g'\in\Gamma}&&V\ar[u]_{\gamma} }
\end{displaymath}
\end{lemma}

\begin{definition}
If $G$ is an algebraic group, then the special subvariety $G_{sp}$ of the three-fold product $G^{3}$ is the set
\[G_{sp}=\{(a,b,c)\in G^{3}, abc=1\}\]
We say that our $F$ is a special subvariety of the product $A\times B\times C$ if there is an algebraic group $G$ and there are multi-functions $(\alpha: G\rightarrow A)$, $(\beta: G\rightarrow B)$ and $(\gamma: G\rightarrow C)$ such that $F$ is a component of the $(\alpha\times\beta\times\gamma)$-image of the special subvariety $G_{sp}\subset G^{3}$. Here $\alpha\times \beta\times\gamma$ is the multi-function naturally induced by $\alpha$, $\beta$, $\gamma$ between $G^{3}$ and $A\times B\times C$.
\end{definition}
One might also understand the special subvariety as parameter separation: the subvariety $F$ is globally ``equivalent" (up to some multi-function map) to the hypersurface $G_{sp}$, where each parameter is separately mapped to corresponding coordinate.
\begin{definition}
Given an algebraic variety $A$ and fix an integer $b$ and the degree $d>0$, we say that a finite collection of points $X\subset A$ is \emph{in general position} if any algebraic subset of degree $\leq d$ and dimension less than $\dim(A)$ contains at most $b$ points of $X$.
\end{definition}

Note that if $A$ is of complex (or real) dimension on, then any finite subset $X$ of $A$ is automatically in general position. 

\begin{theorem}{$\mathbf{(Main}$ $\mathbf{Theorem)}$}\label{thm: main theorem}
Let $A$, $B$, $C$ be projective varieties, and let $F\subset A\times B\times C$ be a subvariety with the property  that the projections $ F \rightarrow A\times B$, $F \rightarrow B\times C$, $F\rightarrow C\times A$ are surjective and generically finite. Write $l=\dim(A)$ when $A$ is a real projective variety, and $l=2\dim(A)$ when $A$ is a complex variety. Then there is a positive constant $\eta=\tfrac{1}{16l-18}-\epsilon$ for any very small $\epsilon>0$, and bounded positive constants $n_{0}=n_{0}(\epsilon)$, $d$ depending on $F$  with the following property: Suppose we choose $n> n_{0}$ points on each variety: $X=\{\mathbf{a}_{1},\dots ,\mathbf{a}_{n}\}\subset A$, $Y=\{\mathbf{b}_{1},\dots,\mathbf{b}_{n}\}\subset B$ and $Z=\{\mathbf{c}_{1}, \dots, \mathbf{c}_{n}\}\subset C$ in general position for degree $d$ and some integer $b$. Assume that $|F\cap (X\times Y\times Z)|\geq n^{2-\eta}$. Then $F$ must be a special subvariety. Moreover, the degrees of the multi-functions relating $A$, $B$ and $C$ to the group are bounded.

\begin{displaymath}
	\xymatrix{F\ar@{^{(}->}[rr] \ar[d]&&A &\times &B &\times &C &&\\
			\mathcal{G}_{sp}\ar@{^{(}->}[rr]&&G\ar[u]_{\alpha} &\times&G\ar[u]_{\beta}&\times&G\ar[u]_{\gamma} }
\end{displaymath}
\end{theorem}

\begin{remark}
We are most interested in the case where $A$, $B$ and $C$ are the real line $\mathbb{R}$ (resp.  complex line $\mathbb{C}$). In this case, we can prove that $\eta=\tfrac{1}{6}$ (resp. $\tfrac{1}{22}-\epsilon$ for any very small $\epsilon$ directly from Elekes-Szab\'{o} paper, and $\tfrac{1}{6}-\epsilon$ by Theorem 3).

\end{remark}
Next lemma is an initial bound we would expect for the intersection between an subvariety $V$ and  a Cartesian product. We will apply it repeatedly in the proof of main theorem to obtain the non-trivial bound $|F\cap (X\times Y\times Z)|\leq n^{2-\eta}$.
\begin{lemma}{($\mathbf{Counting}$ $\mathbf{Lemma})$}
(a) Let $\tilde{A}$ be an algebraic set, and let $\tilde{X} \subset \tilde{A}$  be a finite subset in general position with respect to the family of all algebraic subsets of bounded degree and dimension smaller than $\dim(\tilde{A})$, and $U=\tilde{A}^{r}$ the product of finitely many copies of $\tilde{A}$ (hence $\tilde{X}^{r} \subset U$). Moreover, let $V\subset U$ be a subvariety of bounded degree and "small" dimension: assume that $\dim(V)<(t+1)\dim(\tilde{A})$, for a positive integer $t$. Then
\[|V\cap \tilde{X}^{r}|=\mathcal{O}(|\tilde{X}|^{t}).\]

(b) If $X\subset A$, $Y\subset B$, $Z\subset C$ as in the Main Theorem, $U$ is the product of $r$ terms, each one of $A$, $B$ or $C$, and $S$ is the corresponding $r$-term product of $X$'s, $Y$'s and $Z$'s, then $|V\cap S|=\mathcal{O}(n^{t})$ holds for any $V\subset U$ of bounded degree and dimension smaller than $(t+1)\dim(A)$.
\end{lemma}

\begin{proof}
(a) The idea of the proof is using induction on $\dim(U)$. If $r\leq t$ or $r=1$ then we are done by hypothesis. Otherwise $U$ can be written as $U=\tilde{A} \times U'$ with the natural projection $\pi : V \rightarrow U'$. We separate $U'$ into two parts $U' = U'_{1}\sqcup U'_{2}$ according to dimension of fibres:
\[ U'_{1} :=\{ u\in U' |\dim ( \pi ^{-1} (u))=\dim(\tilde{A})\}; \ \   U'_{2}= U' \setminus U'_{1}.\]
Then $\dim (U'_{1}) \leq \dim(V)-\dim(\tilde{A}) \leq t \dim(\tilde{A})$, and we control contribution of each set by induction on projection images and fibres. 
\[ | V \cap \pi ^{-1}(U'_{1}) \cap \tilde{X}^{r} |= \mathcal{O}(|\tilde{X}|) \cdot \mathcal{O}(|\tilde{X}|^{t-1})=\mathcal{O}(|\tilde{X}|^{t}) \]
\[  |V \cap \pi^{-1}(U'_{2}) \cap \tilde{X}^{r}| =\mathcal {O}(1) \cdot \mathcal{O}(|\tilde{X}|^{t})=\mathcal{O}(|\tilde{X}|^{t})\]

(b) follows easily if we use $A\cup B\cup C$ in place of $\tilde{A}$ and $X \cup Y \cup Z$ in place of $\tilde{X}$ in (a).

\end{proof}

Now we can prove the Main Theorem with the above two lemmas.

\begin{proof}
The projective varieties $A$, $B$, $C$ and algebraic variety $F$ can be assumed irreducible. 

\begin{claim} There exists a dense open set $C'$ of $C$ such that for any $p\in C'$, the fiber $F_{p}$ is a multi-function $A\rightarrow B$. Write $k=dim(A)$, then $F$ is a $k$-dimensional family of multi-functions between $A$ and $B$.

\end{claim}
\begin{proof}
We have natural projections: $\pi_{A} : F_{p}\rightarrow A$ and $\pi_{B} : F_{p} \rightarrow B$. They are surjective because $F\rightarrow A\times C$ and $F \rightarrow B\times C$ are surjective by hypothesis. For a fixed $p\in C$, if $F_{p} \rightarrow A$ is not generically finite, then $\dim(F_{p})>\dim(A)$. Let $C_{0}$ denote the set of points $q \in C$ such that $F_{q}$ is not generically finite, then $\dim(C_{0}) <\dim(C)$. Hence $C_{0}$ lies on a proper closed subvariety of $C$, and we can find a dense open set $C' \subset C$ such that for every point $p$ of $C'$, the fibre $F_{p}$ is a multi-function.

Replacing $C$ with $C'$,  we may assume that any $p\in C$, $F_{p}$ is a multi-function. Next we will prove that $F$ is a $k$-dimensional family of multi-functions. Suppose that the image of $F$ in $\mathcal{H}(A\times B)$, called $R$, is $k' <k$ dimensional. Then for a generic $r\in R$, the inverse image of $r$ is $(k-k')$-dimensional. After switching to a dense open set of $C$ and the corresponding image $R$, we may assume that for all $r\in R$, the inverse image of $r$ is $(k-k')$-dimensional. Since $F\rightarrow A\times B$ is surjective, for every point $(a,b)\in A\times B$, there exists a $c\in C$, such that $(a,b,c)\in F$.  Let $r$ be the image of $c$ in $\mathcal{H}(A\times B$. The inverse image of $r$ is $(k-k')$-dimensional, hence there are infinitely many $c'\in C$ such that $(a,b,c')\in F$, which contradicts the assumption that $F\rightarrow A\times B$ is generically finite.

\end{proof}
 We can replace $C$ with $C'$ while only losing $\mathcal{O}(n)$ points of the form $(\mathbf{a}_{i}, \mathbf{b}_{j}, \mathbf{c}_{k})$ in $F$: we lose at most $\mathcal{O}(1)$ $\mathbf{c}_{k}$ by restricting to $C'$, and at most $n\cdot\mathcal{O}(1)$ of $(\mathbf{b}_{j}, \mathbf{c}_{k})$ in $B\times C$. By hypothesis, $F \rightarrow B\times C$ is generically finite, so we can repeat the trick in the above lemma and conclude that there are at most $\mathcal{O}(n)$ points' loss from $F\cap (X\times Y\times Z)$. 

From now on, for every $p\in C$, $F_{p}$ is a multi-function.

We will consider the following composition of  families of multi-functions $(F_{t} :A\rightarrow B, t\in C)$ and $(F_{s}^{-1}: B\rightarrow A, s\in C)$:

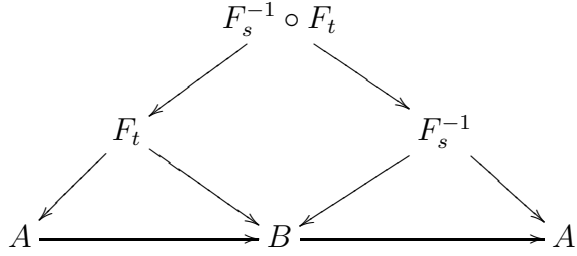
\begin{displaymath}
	\xymatrix{&& F_{s}^{-1}\circ F_{t}\ar[dl]\ar[dr] &&\\
			&F_{t}\ar[dl]\ar[dr]&&F_{s}^{-1} \ar[dl]\ar[dr]&\\
			A \ar[rr]&&B\ar[rr]&&A}
\end{displaymath}

If the family of composition $(F_{s}^{-1} \circ F_{t} : A \rightarrow A, (s,t)\in C\times C)$ has a common component with a $k$-dimensional family of multi-functions, we will use the Composition Lemma to find the group, otherwise we use combinatorical bounds to show that $|F\cap (X\times Y\times Z)|\leq \mathcal{O}(n^{2- \eta})$.

\textbf{Case One.} If $F_{s}^{-1}\circ F_{t}$ has a common component with a $k$-dimensional family, by the Composition Lemma, we have a $k$-dimensional connected algebraic group $\Gamma$ acting on a variety $V$, and multi-functions $(\alpha: A\rightarrow V)$, $(\beta: B\rightarrow V)$ such that the family $F$ is related to the standard family $H$ corresponding to the $\Gamma$-action on $V$ along $\alpha$ and $\beta$. More precisely, the standard family $H$ is $(H_{g}: V\rightarrow V, v \rightarrow g(v), g\in \Gamma)$. We note  $\tilde{F}=\beta\circ H\circ \alpha^{-1}$ the family of multi-functions from $A$ to $B$ parameterized by $\Gamma$.
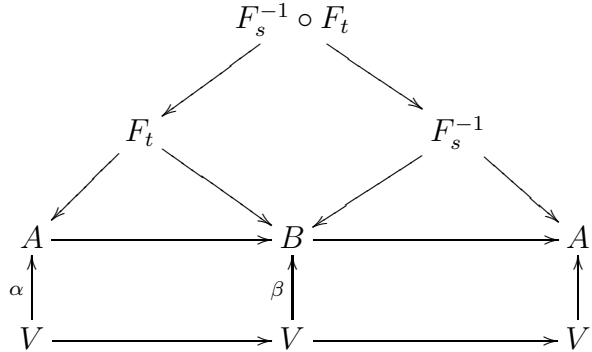
\begin{displaymath}
	\xymatrix{&& F_{s}^{-1}\circ F_{t}\ar[dl]\ar[dr] &&\\
			&F_{t}\ar[dl]\ar[dr]&&F_{s}^{-1} \ar[dl]\ar[dr]&\\
			A \ar[rr]&&B\ar[rr]&&A\\
			V\ar[rr]\ar[u]^{\alpha}&&V\ar[rr]\ar[u]^{\beta}&&V\ar[u]}
\end{displaymath}
Since $F\rightarrow A\times B$ is generically finite and surjective,  the group action $\Gamma$ on $V$ is transitive. And as $\dim (V)$=$\dim(\Gamma)$, algebraic variety $V$ must be the quotient of $G$ by a finite subgroup. We can replace $V$ with $G$, and the new $\alpha$ and $\beta$ are still multi-functions. 

From Composition Lemma, we know that $F$ and $\tilde{F}$ share a common component as families of multi-function:
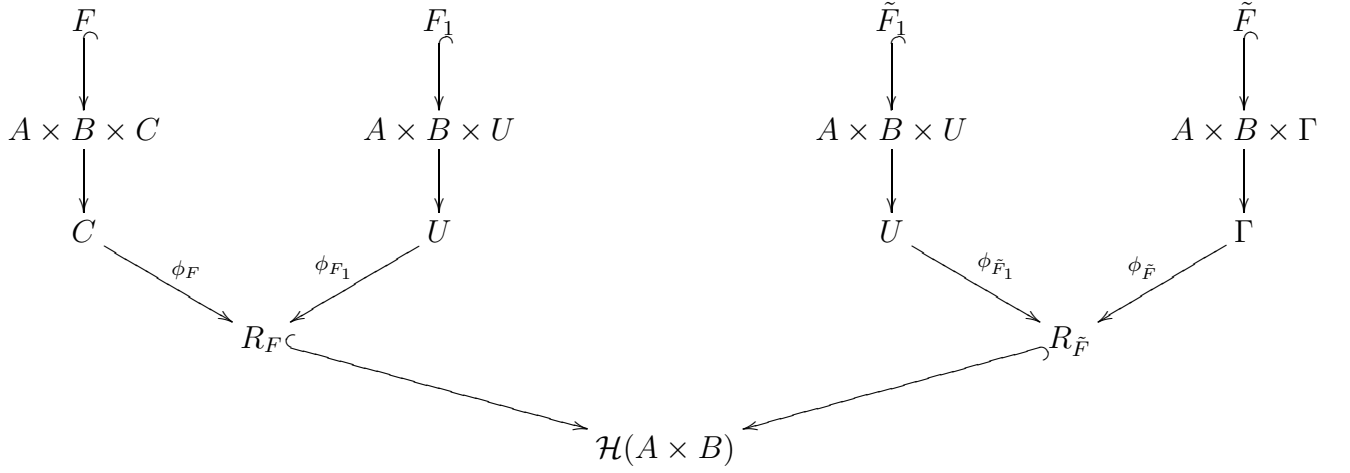
\begin{displaymath}
	\xymatrix{F\ar@{^{(}->}[d] && F_{1} \ar@{^{(}->}[d]&& \tilde{F}_{1}\ar@{^{(}->}[d]&&\tilde{F}\ar@{^{(}->}[d]\\
		A\times B\times C\ar[d] && A\times B\times U\ar[d] &&A\times B\times U\ar[d]&&A\times B\times \Gamma\ar[d]\\
		C\ar[dr]^{\phi_{F}}&&U\ar[dl]_{\phi_{F_{1}}}&&U\ar[dr]^{\phi_{\tilde{F}_{1}}}&& \Gamma\ar[dl]_{\phi_{\tilde{F}}}\\
		&R_{F}\ar@{^{(}->}[drr]&&&&R_{\tilde{F}}\ar@{^{(}->}[dll]&\\
		&&&\mathcal{H}(A\times B)&&&}
\end{displaymath}
Here $R_{F}$ and $R_{\tilde{F}}$ are the images of $C$ and $\Gamma$ in Hilbert scheme $\mathcal{H}(A\times B)$. And $U$ is in fact the $H_{\mu(0,0)}$ in the proof of Composition Lemma, so $\dim U=\dim C$. The dimensions of $C, U, R_{F}, R_{\tilde{F}}, \Gamma$ are the same, hence $\phi_{F}^{-1}, \phi_{F_{1}}, \phi_{\tilde{F}_{1}}^{-1}, \phi_{\tilde{F}}$ are all well-defined multi-functions. Then we obtain a multi-function $\gamma$ from $\Gamma$ to $C$ by compositing them.

One can read from the above two diagrams that roughly $\gamma(c) \oplus \alpha(a)=\beta(b)$. More precisely, to conclude the \textbf{Case One}, we show that $F$ is a component (as algebraic variety) of the corresponding multi-function image of $\tilde{F}$. 

If almost every fiber $F_{c}, c\in C$ is a irreducible variety, then we know that $F_{1}$ is a subvariety of $\tilde{F}_{1}$ (at least in a dense open set). As $\tilde{F}$ is $(\alpha\times\beta\times Id)$-image of $\Gamma_{sp}$, and $F$ is irreducible by hypothesis, we conclude that $F$ is a component of the image of $\Gamma_{sp}$ under $(\alpha\times\beta\times\gamma)$.

Otherwise one can apply Claim~\ref{cl: composition lemma trick}, and replace $F$ with another family of multi-functions $F'_{0}$, which has irreducible fibers. In addition, the morphism $F_{0}\rightarrow F$ is dominant.  As shown above, $F_{0}$ is a component of multi-function image of $\Gamma_{sp}$, so is $F$.

\textbf{Case Two.} Any component of the family $\mathcal{F} :=\{F_{s}^{-1} \circ F_{t}\}$ is at least $(k+1)$-dimensional family.

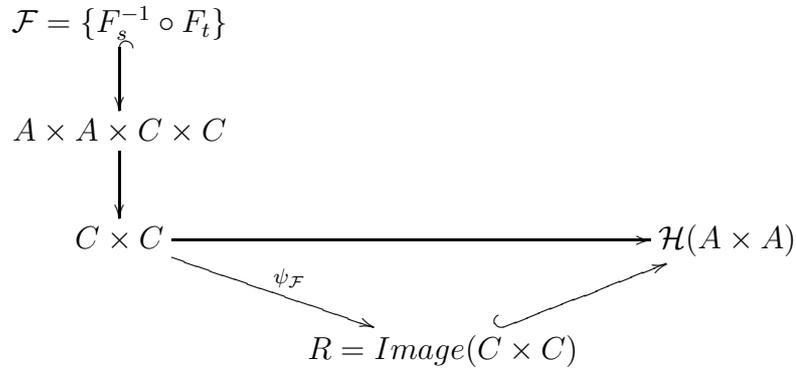
\begin{displaymath}
	\xymatrix{\mathcal{F}=\{F_{s}^{-1}\circ F_{t}\} \ar@{^{(}->}[d]&&\\
		A\times A\times C\times C\ar[d]&&\\
		C\times C\ar[dr]^{\psi_{\mathcal{F}}}\ar[rr]&&\mathcal{H}(A\times A)\\
		&R = Image(C\times C)\ar@{^{(}->}[ur] &}
\end{displaymath}

We shall use the Claim~\ref{cl: composition lemma trick} from the proof of Composition Lemma to shrink $C\times C$ (we still have $\mathcal{O}(n^{2-\eta})$ points after shrinking) and split $\mathcal{F}$ into irreducible components $\mathcal{F}^{i}$ such that the fibres $\mathcal{F}^{i}_{st}$ of each component are irreducible. Without loss of generality, we consider only one irreducible component in the following arguments.

Let $\mathcal{F}_{lm}$ denote the multi-function parameterized by $(\mathbf{c}_{l}, \mathbf{c}_{m})\in Z\times Z \subset C\times C$. The point $(\mathbf{c}_{l},\mathbf{c}_{m})$ is mapped to a point $x$ in the Hilbert scheme $\mathcal{H}(A\times A)$, and $x$ parametrizes exactly the closed sub variety $\mathcal{F}_{lm}$ of $A\times A$. 

We will apply the Counting Lemma to estimate the number of points in the set 
$\mathcal{H}=\{(\mathbf{a}_{i},\mathbf{a}_{j},\mathbf{c}_{l},\mathbf{c}_{m});(\mathbf{a}_{i},\mathbf{a}_{j})\in\mathcal{F}_{lm}\}$. 
Let $\mathcal{U}$ denote the union of those fibers of $\psi$ whose dimension less than $\dim(A)$, which is a dense open set of $C\times C$. Let $\mathcal{V}$ be the complement of $\mathcal{U}$. We shall call the points $(\mathbf{c}_{l},\mathbf{c}_{m}) \in \mathcal{V}\cap Z^{2}$ and corresponding $\mathcal{F}_{lm}$  forbidden, and call the others ordinary. We will bound the cardinality of $\mathcal{H}$ separately for the forbidden set and the ordinary set. 

As $\mathcal{V}$ is a proper closed set of $C\times C$, its dimension is strictly smaller than $\dim(C\times C)$. Applying Counting Lemma again, we have $|\mathcal{V}\cap Z^{2}| \leq \mathcal{O}(n)$, i.e. there are at most $\mathcal{O}(n)$ forbidden points. Hence there are at most $\mathcal{O}(n)$ forbidden multi-functions. Moreover, on  the closed subvariety defined by multi-function $\mathcal{F}_{lm}$, there are at most $\mathcal{O}(n)$ points  $(\mathbf{a}_{i}, \mathbf{a}_{j})$ of $X\times X $. As such, the contribution of forbidden multi-functions on $\mathcal{H}$ is at most $\mathcal{O}(n^{2})$.

One can switch the role of $A$ and $C$ and repeat the above arguments until now. The new composition family multi-functions $\tilde{\mathcal{F}}$ parameterized by $A\times A$ is indeed the same subvariety $\mathcal{F}$ of $A\times A\times C\times C$. To say that a point $(\mathbf{c}_{l}, \mathbf{c}_{m})$ lies on multi-function $\tilde{\mathcal{F}}_{ij}$ is the same as to say that the point $(\mathbf{a}_{i},\mathbf{a}_{j})$ lies on multi-function $\mathcal{F}_{lm}$.  Similarly, we have at most $\mathcal{O}(n)$ of $(\mathbf{a}_{i}, \mathbf{a}_{j})$ corresponding to forbidden multi-functions (from $C$ to $C$), and at most $\mathcal{O}(n^{2})$ points $(\mathbf{a}_{i},\mathbf{a}_{j},\mathbf{c}_{l},\mathbf{c}_{m})$ come from forbidden multi-functions(from $C$ to $C$).

Next we will bound the contribution of ordinary points to $\mathcal{H}$, i.e. when both $(\mathbf{a}_{i},\mathbf{a}_{j})$ and $(\mathbf{c}_{l},\mathbf{c}_{m})$ are ordinary. There are at most $\mathcal{O}(n^{2})$ ordinary multi-functions, each multi-function is irreducible by assumption. And every ordinary multi-function corresponds to at most $\mathcal{O}(1)$ points from $X\times X$ or $Z \times Z$. From now on, we do not distinguish ordinary multi-function $\mathcal{F}_{l,m}$ from the point $(\mathbf{c}_{l},\mathbf{c}_{m})$ representing it. Every two multi-functions (subvarieties) intersect in at most $\mathcal{O}(1)$ points in $X\times X$ (resp. $Z\times Z$), or we can say that through every two ordinary points from $X\times X$ (resp. $Z \times Z$), there are at most $\mathcal{O}(1)$ multi-functions parameterized by points in $Z\times Z$ (resp. $X\times X$ ). 

Estimate size of $\mathcal{H}$ turns out to estimate incidences between points and algebraic varieties. In the most general case, the underlying point set is of combinatorial dimension 2 with respect to the underlying algebraic varieties set. Then apply Incidence Theorem ~\ref{thm: incidence theorem} with $D=4l-4$ and $k=2$, we have
\begin{equation}\label{eq: incidence bound}
|\mathcal{H}|\leq \mathcal{O}(n^{3-\eta'}) \text{ for }\eta' = \tfrac{1}{8l-9}-\epsilon
\end{equation}
for any very small $\epsilon$.

Next we want to get a lower estimate (which relate to $|F\cap X\times Y\times Z|$)  and compare with it.
 Let $\mathcal{H}'=\{(\mathbf{a}_{i},\mathbf{a}_{j},\mathbf{b}_{k},\mathbf{c}_{l},\mathbf{c}_{m}); (\mathbf{a}_{i},\mathbf{b}_{k})\in F_{c_{l}}, (\mathbf{a}_{j},\mathbf{b}_{k})\in F_{c_{m}} \}$ , 
 $W'=\{(\mathbf{a}, \mathbf{a}',\mathbf{b}, 
 \mathbf{c},\mathbf{c}'); (\mathbf{a}, \mathbf{b}, \mathbf{c}) , (\mathbf{a}', \mathbf{b}, \mathbf{c}')\in F\}$, 
  $W=\{(\mathbf{a},\mathbf{a}',\mathbf{c},\mathbf{c}'); \exists \mathbf{b}: (\mathbf{a}, \mathbf{b}, \mathbf{c}) , (\mathbf{a}', \mathbf{b}, \mathbf{c}')\in F\}$. Clearly $\mathcal{H}' \subset W'$ and $\mathcal{H}\subset W$.  There is a natural projection $\phi: W'\rightarrow W$. The dimension of $W'$ is $3 \cdot\dim(A)$ because we can choose $\mathbf{a}$, $ \mathbf{b}$, and $\mathbf{c}$ freely in $F$ (of dimension $2\cdot \dim(A)$) and $\mathbf{a}'$ freely in $A$, then there is only $\mathcal{O}(1)$ many $\mathbf{c}'$ satisfying the condition. The same argument implies also $\dim(W)=3\cdot\dim(A)$.

Apply the Counting Lemma again. Let $\bar{V}\subset W'$ be the union of $\dim(B)$ fibers under the projection $\phi$, and $\bar{U}=W\setminus \bar{V}$. Since $\dim(\bar{V})<\dim(W)=3 \cdot \dim(A)$, we have  $|\bar{V}\cap \mathcal{H}'|=\mathcal{O}(n^{2})$ by Counting Lemma. Since $\bar{U}$ is the union of fibers less than $\dim(B)$-dimensional,  \[|\bar{U}\cap \mathcal{H}'|\leq \mathcal{O}(1)\cdot|\mathcal{H}|\leq \mathcal{O}(n^{3-\eta'}).\]

On the other hand, 

\begin{align*}
|\mathcal{H}'|&=\sum_{\mathbf{b}_{k}\in Y} |\{(\mathbf{a}_{i}, \mathbf{a}_{j}, \mathbf{b}_{k},\mathbf{c}_{l}, \mathbf{c}_{m}); (\mathbf{a}_{i}, \mathbf{b}_{k}, \mathbf{c}_{l})\in F, (\mathbf{a}_{j},\mathbf{b}_{k},\mathbf{c}_{m})\in F\}| \\
&=\sum_{\mathbf{b}_{k}\in Y} |\{(\mathbf{a}_{i},\mathbf{b}_{k},\mathbf{c}_{l})\in F\}| \cdot |\{\mathbf{a}_{j},\mathbf{b}_{k},\mathbf{c}_{m})\in F\}| \\
&=\sum_{\mathbf{b}_{k}\in Y} |\{(\mathbf{a}_{i},\mathbf{b}_{k},\mathbf{c}_{l})\in F\}|^{2}\\
& \geq \tfrac{1}{n} (\sum_{\mathbf{b}_{k}\in Y}|\{(\mathbf{a}_{i},,\mathbf{b}_{k},\mathbf{c}_{l})\in F\}|^{2} =\tfrac{1}{n} |F\cap X\times Y\times Z|^{2}
\end{align*}

Inserting the estimations of $|\mathcal{H}'|$ on the left, we have $\mathcal{O}(n^{3-\eta'}) \geq \tfrac{1}{n}|F\cap X\times Y\times Z|^{2}$. So we have the bound $|F\cap X\times Y\times Z| \leq \mathcal{O}(n^{2-\eta})$, where $\eta = \tfrac{1}{2} \eta'=\tfrac{1}{16l-18}-\epsilon$ (for any small $\epsilon$).

\end{proof}

The Theorem ~\ref{thm: main theorem} works for all algebraic closed fields. For those fields not algebraically closed, we could extend to their algebraic closures and then take restriction. And the exponent on $n$ depends only on the incidence estimate ~\ref{eq: incidence bound}. As long as we could get a better incidence bound in specific cases,  we would improve the bounds on the Theorem ~\ref{thm: main theorem}.

For example, if $A$, $B$, $C$ are real lines $\mathbb{R}$, the underlying multi-functions are indeed bounded degree algebraic curves in the plane $\mathbb{R}^{2}$. Moreover, the underlying point-curve incidences satisfy ``two degrees of freedom and multiplicity-type $\mathcal{O}(1)$" condition in Pach-Sharir paper \cite{pach1998number} : for any two points there are at most $\mathcal{O}(1)$ curves of $L$ passing through them, and any pair of curves from $L$ intersect in at most $\mathcal{O}(1)$. From \cite{pach1998number}, we have $|\mathcal{H}|\leq \mathcal{O}(n^{8/3})$, thus we obtained $n^{11/6}$ in Theorem ~\ref{thm: main theorem} and Theorem ~\ref{thm: elekes-szabo} in real case. 

If $A$, $B$ and $C$ are complex lines $\mathbb{C}$, we shall apply the following Theorem ~\ref{thm: complex}. Then we improve the exponent of $n$ from $\tfrac{21}{22}+\epsilon$ to $\tfrac{11}{6}+\epsilon$.

\chapter{A bound on complex case}\label{section: complex}

\begin{theorem} \label{thm: complex}
Let $\mathcal{F}\subset A\times A\times C\times C$ be a family of multi-functions as in the above Chapter. Moreover, suppose that $A=C=\mathbb{C}$. Then each multi-function $\mathcal{F}_{c,c'}\subset A\times A$ is in fact a complex algebraic curve of bounded degree, and so is $\mathcal{F}_{a,a'}$ for any $(c,c')\in C\times C$ and any $(a,a')\in A\times A$. There is a bounded number $d$ depending only on $\mathcal{F}$, and a bounded positive constant $n_{0}$ with the following property: We choose $n$ points as in the Main Theorem: $X=\{\mathbf{a}_{1},...,\mathbf{a}_{n}\}\subset A$ and $Z=\{\mathbf{c}_{i},...,\mathbf{c}_{n}\}\subset C$ Then$|\mathcal{F}\cap (X\times X\times Z\times Z)|$ is at most $\mathcal{O}(n^{\tfrac{8}{3}+\epsilon})$ for any small $\epsilon>0$.
\end{theorem}

Our main observation is that in a three-dimensional real algebraic surface (the zero set of the partition polynomial), we can get rid of the transversality condition in the paper of Solymosi and Tao \cite{solymosi2012incidence}.

\begin{proof}
Using the convention in the above chapter, we assume that the incidences we count happens between ordinary multi-functions and ordinary points. As before, $\mathcal{F}_{c.c'}$ represents the fiber of $\mathcal{F}$ over $(c, c')$, and the fiber of $\mathcal{F}$ over $(\mathbf{c}_{l}, \mathbf{c}_{m})$ is written by $\mathcal{F}_{l,m}$ for short. We do not distinguish an ordinary multi-function $\mathcal{F}_{l,m}$ from$(\mathbf{c}_{l}, \mathbf{c}_{m})$, since one multi-function $\mathcal{F}_{l,m}$ corresponds to at most $\mathcal{O}(1)$ points in $Z \times Z$.

\textbf{First Step.}

The complex dimension of $\mathcal{F}$ is three because $\mathcal{F}_{c,c'}$  is a complex algebraic curve in $\mathbb{C}^{2}$ of dimension one for a generic point $(c,c')\in C\times C$.  There is a constant degree irreducible polynomial $f(x_{1},x_{2},z_{1},z_{2})$ in the ideal of $\mathcal{F}$. So for a fixed $(c,c')\in C\times C$, the zero set of $f_{c,c'}:=f(x_{1},x_{2},c,c')$ contains $\mathcal{F}_{c,c'}$.

If two curves $\mathcal{F}_{c_{1},c_{1}'}$ and $\mathcal{F}_{c_{2},c_{2}'}$ have non transversal intersection at a point $(a,a')\in A\times A$, then the zero sets of $f_{c_{1},c_{1}'}$ and $f_{c_{2},c_{2}'}$ intersect at $(a,a')$ non-transversally. In other words, we must have $\partial_{x_{1}}f_{c_{1},c_{1}'}\partial_{x_{2}}f_{c_{2},c_{2}'} - \partial_{x_{2}}f_{c_{1},c_{1}'}\partial_{x_{1}}f_{c_{2},c_{2}'}=0$. 

One can assume that $X\times X$ contains no singular points of the collection of curves $\{ \mathcal{F}_{l,m}, (\mathbf{c}_{l}, \mathbf{c}_{m})\in Y\times Y\}$ because the following: the collection of all $\{(\mathbf{a}_{i}, \mathbf{a}_{j},\mathbf{c}_{l},\mathbf{c}_{m})\}$ such that $ (\mathbf{a}_{i}, \mathbf{a}_{j})$ is a       singular point of $\mathcal{F}_{l,m}$ is the subset of $\{f=0\} \cap \{\partial_{x_{1}}f=0\}\cap\{\partial_{x_{2}}f=0\} \cap X\times X\times Y\times Y$. The dimension of $\{f=0\} \cap \{\partial_{x_{1}}f=0\}\cap\{\partial_{x_{2}}f=0\}$ is strictly smaller than 3 since $f$ is chosen irreducible, by the Counting Lemma, incidences happen on singular points of curves less than $\mathcal{O}(n^{2})$ times.

\textbf{Second Step.}

We will rule out some but not necessary all transversal intersections as well. We shall adapt the proof of Solymosi-Tao's theorem:

Let $P=X\times X$ denote the collection of points, and let $L$ denote the collection of curves $\mathcal{F}_{l,m}$ for all  $(\mathbf{c}_{l}, \mathbf{c}_{m})\in Z\times Z$. For the sake of simplicity, we shall call any point in $P$ by $p$, and any curve in $L$ by $l$. In the following, $C$ will be a large constant depending on degree $d$ and $\epsilon$,   $C_0, C_1$ and $C_2$ will be positive constants to be chosen later: $C_{0}, C_1 >2$ are sufficiently large depending on $\epsilon$ , $d$ and $C$, and $C_{2}$ will be sufficiently large depending on $C_{1}$, $C_{0}$, and $\epsilon$.

 We perform an induction on $|P|$. Suppose that for $|P'|\leq \tfrac{|P|}{2}$ and $|L'| \leq |L|$, we already have
\begin{equation}\label{hypothesis}
|\mathcal{I}(P',L')|\leq C_{2}|P'|^{\tfrac{2}{3}+\epsilon}|L'|^{\tfrac{2}{3}}+C_{0}(|P'|+|L'|).
\end{equation}

Our goal is to prove
\begin{equation}\label{inductiveclaim}
|\mathcal{I}(P,L)|\leq C_{2}|P|^{\tfrac{2}{3}+\epsilon}|L|^{\tfrac{2}{3}}+C_{0}(|P|+|L|).
\end{equation}

We apply polynomial cell decomposition to $D = C_{1}$ on $\mathbb{C}^2 \simeq \mathbb{R}^{4}$ and obtain a partition:
\begin{equation}
\mathbb{R}^{4}=\{Q=0\}\cup U_{1}\cup\cdots \cup U_{M}.
\end{equation}

Here $Q: \mathbb{R}^{4}\rightarrow\mathbb{R}$ has degree at most $C_{1}$, so $M \sim C_{1}^{4}$ and $|P_{i}| = |P \cap U_{i}| \leq \tfrac{|P|}{2}$ for large $C_{1}$. We denote $L_{i}$ to be the set of curves in $L$ having nonempty intersection with $U_{i}$, and $P_{cell}$ to be the set of points lying in one of the cells, $L_{cell}$ to be the the union of all $L_{i}$. On the other side, let $P_{alg}$ denote the set of points lying on zero set of $Q$, and let $L_{alg}$ denote the set of curves lying completely on zero set of $Q$. By induction hypothesis,
\begin{align*}
|\mathcal{I}(P_{i},L_{i})| &\leq  C_{2}|P_{i}|^{\tfrac{2}{3}+\epsilon}| L_{i}|^{\tfrac{2}{3}}+C_{0}(|P_{i}| + |L_{i}|)\nonumber\\
&\leq  C^{\tfrac{2}{3}+\epsilon} C_{2}C_{1}^{-4(\tfrac{2}{3}+\epsilon)}|P|^{\tfrac{2}{3}+\epsilon}|L_{i}|^{\tfrac{2}{3}}+C_{0}(|P_{i}|+|L_{i}|)).
\end{align*}

For $l$ belonging to some $L_{i}$, we apply a result in real algebraic geometry which implies the number of connected components of $l\setminus \{Q=0\}$ is at most $CC_{1}^{2}$ (see Theorem A.2 of \cite{solymosi2012incidence}). From which we have
\begin{equation}
\sum_{i=1}^{M}|L_{i}| \leq CC_{1}^{2}|L|
\end{equation}

Add up $|\mathcal{I}(P_i,L_{i})|$ and apply H\"{o}lder Inequality, we obtain the following estimation:
\begin{align*}
|\mathcal{I}(P_{cell},L_{cell})| &=  \sum_{i=1}^{M}|\mathcal{I}(P_{i},L_{i})|\nonumber\\&
\leq  C^{\tfrac{2}{3}+\epsilon}C_{2}C_{1}^{-4(\tfrac{2}{3}+\epsilon)}|P|^{\tfrac{2}{3}+\epsilon}(\sum_{i=1}^{M}|L_{i}|^{\tfrac{2}{3}})+C_0 (|P|+CC_{1}^{2}|L|))\nonumber\\ &
\leq C^{\tfrac{2}{3}+\epsilon}C_{1}^{-4\epsilon}C_{2}|P|^{\tfrac{2}{3}+\epsilon}|L|^{\tfrac{2}{3}}+C_{0}(|P|+CC_{1}^{2}|L|))
\end{align*}

Now we establish two initial bounds from the fact that  every two curves intersect in at most $\mathcal{O}(1)$ points, and there are at most $\mathcal{O}(1)$ curves  through every pair of points.  Suppose that the implicit constants here are smaller than $C$: 

\begin{equation}\label{L2+P}
|\mathcal{I}(P,L)|\leq C |L|^{2}+|P|.
\end{equation}

\begin{equation}\label{PA+L}
|\mathcal{I}(P,L)|\leq C |P|^{2}+|L|.
\end{equation}

Hence we may assume that $|P|^{\tfrac{1}{2}}\leq|L|\leq|P|^{2}$, otherwise $|\mathcal{I}(P,L)|\lesssim_{d}|P|+|L|$ and it suffices to choose $C_0$ bigger than the implicit constant. With this assumption,

\begin{equation}\label{PcellLcell}
|\mathcal{I}(P_{cell},L_{cell})| \leq  (C^{\tfrac{2}{3}+\epsilon}C_{1}^{-4\epsilon}C_{2}+C_{0}|P|^{-\epsilon}+CC_{1}^{2}|P|^{-\epsilon})|P|^{\tfrac{2}{3}+\epsilon}|L|^{\tfrac{2}{3}}.
\end{equation}

Choose $C_{1}$ so that $C^{\tfrac{2}{3}+\epsilon}C_{1}^{-4\epsilon}$ is smaller than $\tfrac{1}{2}$, and it is enough to conclude the induction with a sufficiently large $C_{2}$ provided that
\begin{equation}\label{PalgL}
\mathcal{I}(P_{alg}, L)\lesssim_{C_1}|P|^{\tfrac{2}{3}}|L|^{\tfrac{2}{3}} + |P|+|L| 
\end{equation}

\textbf{Third Step.} 
Write $D=C_{1}$ and $\Sigma =\{Q=0\}$. By an algebraic geometry result (see for example corollary 4.5 in \cite{solymosi2012incidence}), we can decompose $\Sigma$ into smooth into smooth points on subvarieties: 
\[\Sigma = \Sigma^{smooth}\cup\Sigma_{i}^{smooth}.\]

Here $\Sigma_{i}$ are subvarieties in $\Sigma$ of dimension smaller than 3 and degree $\mathcal{O}_{D}(1)$, and the number of $\Sigma_{i}$'s is at most $\mathcal{O}_{D}(1)$.  It suffices to show that for any algebraic variety $\Sigma$ of dimension smaller or equal 3, degree at most $\mathcal{O}_{D}(1)$, the intersection happens on its smooth points is bounded by $|\mathcal{I}(\Sigma^{smooth}\cap P, L)| \lesssim_{C_{1}} |P|^{\tfrac{2}{3}}|L|^{\tfrac{2}{3}}+|P|+|L|$

Let $L_{alg}$ denote the set of curves in $L$ that lie on $\Sigma$, let $L_{non-alg}$ denote the set of curves not lying on $\Sigma$ but have non-empty intersection. We first estimate $|\mathcal{I}(\Sigma^{smooth}\cap P, L_{non-alg})|$.

\textbf{Estimation of incidences on $L_{non-alg}$}

If $l$ does not lie in $\Sigma$, then by corollary 4.5 of \cite{solymosi2012incidence} again, the intersection between $l$ and $\Sigma$ has only constantly many components:  $l\cap\Sigma=\cup_{j=0}^{J}l_{j}$ for some $J = O_{D}(1)$ (we allow empty component to achieve a uniform number of components).  And for each $1\leq j\leq J$, $l_{j}$ is an algebraic variety of (real) dimension smaller than one and of degree $O_D (1)$. Let  $\mathcal{I}_{j}$ denote the incidences between points and $j$th component of $l$, for all $l\in L_{non_alg}$. So we only need to bound each $\mathcal{I}_{j}$ separately, and 
\begin{equation}
|\mathcal{I}(P\cap \Sigma^{smooth}, L_{non-alg})|\leq \sum_{ j \leq J} |\mathcal{I}_{j}|
\end{equation}
Notice that for each $l_{j}$, if it is not the union of $\mathcal{O}_{D}(1)$ points, then it belongs to a unique $l$ because the intersection of $l\in L$ and $l'\in L$ has dimension 0. Now take a generic projection from $\mathbb{R}^{4}$ to $\mathbb{R}^{2}$, and use the Pach-Sharir bound on real plane curves, we have 
\[|\mathcal{I}(P\cap \Sigma^{smooth}, L_{non-alg})|\lesssim_{D} |P|^{\tfrac{2}{3}}|L|^{\tfrac{2}{3}}+|P|+|L|.\]

\textbf{Estimation of incidences on $L_{alg}$}

We will show that through each point $p\in P\cap \Sigma^{smooth}$, there are at most $\mathcal{O}_{D}(1)$ curves in $L_{alg}$ passing through it. This is the new ingredient we add to Solymosi and Tao's proof \cite{solymosi2012incidence}. Fix a point $p\in P\cap \Sigma^{smooth}$, if there are two curves $l_{1}$ and $l_{2}$ of $L_{alg}$ intersect at $p$, then the intersection cannot be transversal on the three dimensional variety $\Sigma$.   
On the other side, $l_{1}$ and $l_{2}$ are complex curves, they intersect non-transversally means that they have the same complex tangent vector on $p$. So for a given $p$, all complex curves in $L_{alg}$ passing through it share a common complex tangent vector $\mathbf{v}_{p}$. 

Back to our assumptions, $p $ is of the form $(\mathbf{a}_{i}, \mathbf{a}_{j})\in X\times X$, and $l$ is of the form $\mathcal{F}_{l,m}$, which is also a component of $\{f(x_{1}, x_{2}, \mathbf{c}_{l},\mathbf{c}_{m})=0\}$ in \textbf{First Step}. 

The number of curves $l$ passing through $p=(\mathbf{a}_{i},\mathbf{a}_{j})$ with an underlying complex tangent vector $\mathbf{v}_{p}=(v_{p,1}, v_{p,2}) $ is smaller than the number of points of $Z\times Z$ lying on the algebraic variety parametrized by the following equations:
\begin{align*}
&f(\mathbf{a}_{i},\mathbf{a}_{j}, z_{1}, z_{2})=0,\\
&v_{p,2} \partial_{x_{1}}f(\mathbf{a}_{i}, \mathbf{a}_{j}, z_{1}, z_{2})-v_{p,1}\partial_{x_{2}}f(\mathbf{a}_{i},\mathbf{a}_{j}, z_{1},z_{2})=0;
\end{align*}

As $f$ is assumed to be irreducible, the algebraic variety parametrized by the above equations is of complex dimension less than two. We apply the Counting Lemma and conclude that there are at most $\mathcal{O}(1)$ points of $Z\times Z$ lying on it. 

From the above arguments, there are no more than $\mathcal{O}(1)$ curves in $L_{alg}$ passing through every point $p\in P$, and we have the estimation:
\[|\mathcal{I}(P\cap \Sigma^{smooth}, L_{alg})|\leq \mathcal{O}(1)|P|.\]
Adding them up we obtain :
\[|\mathcal{I}(P\cap \Sigma^{smooth}, L)|\lesssim_{C_{1}}|P|^{\tfrac{2}{3}}|L|^{\tfrac{2}{3}}+|P|+|L|\]
which concludes our induction.

\end{proof}
%

\chapter{Incidence Theorems from Elekes-Szab\'{o}}\label{section: incidence theorem}
In this chapter, we will introduce the notion of combinatorics dimension, a generalization of certain types of incidence problems: points-curves incidences, points-varieties incidences, etc. 
\begin{definition}
For a fixed constant $b$, let $G \subset S\times T$ be a finite bipartite graph. We say that $S$ has combinatorial dimension 0 in $G$ if $S$ has at most $b$ vertices. In general, for $k\geq 1$, $S$ has combinatorial dimension at most $k$, if there is a subset $T' \subset T$ such that 

(i) $T'$ is ``almost the whole" of $T$, ie. $|T\setminus T'| \leq b$, and

(ii) for all $t\in T'$ the subset $S_{t}$ has combinatorial dimension at most $k-1$ in the subgraph $G(S_{t}, T'\setminus {t})$.
\end{definition}

\begin{remark}
At first glance, one might find it a bit abstract to understand this recursive definition. We might have more intuition with these two examples: incidences between points and degree $d$ algebraic curves, which is of combinatorics dimension 2; and complete bipartite subgraph, which should have large combinatorics dimension.
\end{remark}

We give some basic properties of combinatorics dimension without proof(use induction):

\begin{proposition}
Let $G\subset S\times T$ be a bipartite graph and assume that $S$ has combinatorial dimension $k\geq 1$. Then:

(i) in each subgraph of $G$ the corresponding subset of $S$ has combinatorial dimension at most $k$.

(ii) each complete bipartite subgraph of $G$ has at most $\mathcal{O}(|S|+|T|)$ edges, and

(iii) $G$ has at most $\mathcal{O}(|S|+|S|^{1-\tfrac{1}{k}}|T|)$ edges.

The constants in these big-$O$'s expressions depend on $k$ and $b$, but not on the graph $G$.
\end{proposition}

\begin{theorem}{ $\mathbf{(Incidence} $ $ \mathbf{Theorem)}$}\label{thm: incidence theorem}
Let there be given a family of algebraic subsets $\mathcal{F}$ of a complex projective space $\mathbb{C}P^{N}$, parametrized by an algebraic set $Y$ (of some other projective space). Let $\mathcal{V}$ be a finite sub-collection from this family, and $\mathcal{P}$ a finite point set of combinatorial dimension $k$ with respect to $\mathcal{V}$. Then there exists a constant $D = 4\cdot\dim(Y) -4>0$ such that, for any $\epsilon$ with 
\[0<\epsilon<\frac{k-1}{k(Dk-1)}\]
and values
\[ \alpha := \frac{D(k-1)}{Dk-1}-\epsilon;\]
\[\beta := k(1-\alpha)= \frac{k(D-1)}{Dk-1}+k\epsilon,\]
we have 
\[I(\mathcal{P}, \mathcal{V})=\mathcal{O}\big( |\mathcal{P}|^{\alpha}|\mathcal{V}|^{\beta}+|\mathcal{P}|+|\mathcal{V}|\log(2|\mathcal{P}|) \big),\]
\end{theorem}
The constant of this big-$O$'s expression depends on $b$, $k$, $\epsilon$, $\dim(Y)$, $deg(Y)$, $N$, and the maximum degree of the members of the family (which is finite in each algebraic family).

\begin{remark}
The constant $D$ depends on the fact that $Y$ can be generically imbedded in $\mathbb{R}^{\tfrac{D+2}{2}}$. If the parameter space is simple, we may obtain better constant $D$.
 For example, when $Y=\mathbb{R}^{2}$ (resp. $\mathbb{C}^{2}$), $D$ can be reduced to 2 (resp. 6).
\end{remark}

\begin{proof}

\textbf{First Step. } We represent every algebraic set in the family $\mathcal{F}$  by the corresponding point in parameter space $Y$,  and every point $y \in Y$ corresponds to the algebraic set $y*$ in family $\mathcal{F}$.  And to each point $p\in P$, we assign the corresponding algebraic subset of $Y$ as following:
\begin{align*}
 \mathcal{V}\in & \mathcal{F}  \longleftrightarrow Y &   &\mathcal{P} \longleftrightarrow \{\text{algebraic subset}\} \\
& y \leftrightarrow  y* &    & p  \leftrightarrow \{y\in Y| p\in y*\}\\
&\mathcal{V} \leftrightarrow T &    & \mathcal{P} \leftrightarrow S
\end{align*}

And we denote by $S$ the set of algebraic subsets assigned to points in $\mathcal{P}$, and by $T$ the set of points of $Y$  assigned to the original algebraic sets in the family $\mathcal{F}$. We consider the dual point-algebraic sets incidence problem.

\textbf{Second Step}. Given $s$ hyperplanes in $\mathbb{R}^{d}$ and any positive integer $r<s$, the $\mathbb{R}^{d}$ space can be subdivided (semi-cylindrical decomposition) into $leq r^{d}$ parts such that each part is cut by $B (s \log r /r)$ of the algebraic subsets, for certain constant $B$ depending on $d$. (\cite[Theorem 4.2]{chazelle1991singly})

\textbf{Third Step}.
We put $s=|S|$ and $t=|T|$. And we fix a sufficiently large constant $r$ (to be decided later) and apply the \textbf{Second Step}:

For general case, $Y$ is an algebraic set in $\mathbb{CP}^{N}$. Without loss of generality, we may assume that T have no intersection with hyperplane at $\infty$. And we identify $\mathbb{C}^{N}$ with $\mathbb{R}^{2N}$. We write $d=\dim(Y)$ the dimension of $Y$ in $\mathbb{R}^{2N}$ and project $Y$ to $\mathbb{R}^{2d-1}$ generically such that no incidences would be lost, and the algebraic subsets in $S$ will become real algebraic sets of dimension at most $2d-2$. 

Then we use cutting lemma in Step two. If $d>1$, the algebraic subsets in $S$ lie in $\mathbb{R}^{2d-1}$ and we divide the space $\mathbb{R}^{2d-1}$ into $\leq r^{4d-4}$ parts. In the following, we write $D$ as the parameter appearing in cutting lemma that the space is divided into $\leq r^{D}$ parts. In the most general case, we have $D=4d-4$; while in some specific cases, we might get better $D$. For example, if $d\leq 1$, then we can use cutting lemma for hyperplanes and $D=d$.

In all cases we have a decomposition into $r^{D}$ parts, and each part is cut by $\mathcal{O}(s \log r /r)$ of the algebraic subsets in $S$: $\mathbb{R}^{d}=\sqcup_{i=1}^{M} \mathcal{C}_{i}$.And we call each $\mathcal{C}_{i}$ a cell. 

\textbf{Fourth Step.}
Recall that $r$ is a sufficiently large constant to be determined. If s is very small or very large regarding $r$, our bound 

          $I(\mathcal{P}, \mathcal{V})=I(S,T) \leq C'(|S|^{\alpha}|T|^{\beta} +|S|+|T|\log(2|S|)) $

is valid for certain choice of $C'$ depending on $r$:

\textbf{(1)} If $s\leq r$, then we have $I(S,T) \leq st \leq rt$ when $C' \geq r$. 

\textbf{(2)} If $s\geq r^{-\tfrac{D}{1-\alpha}}t^{k}$, then as it is shown in Proposition, $I(S,T)\leq \mathcal{O}(|S|+|S|^{1-\tfrac{1}{k}}t)\leq \mathcal{O}(1+r^{\tfrac{D}{\beta}})s$. It is enough to choose $C'$ larger than some constant times $1+r^{\tfrac{D}{\beta}}$, which is still a constant.

\textbf{Fifth Step.}
For general $s$ between $r$ and $r^{-\tfrac{D}{1-\alpha}}t^{k}$, then $s\leq r^{-D}s^{\alpha}t^{\beta}$, i.e. $s$ is dominated by $ r^{-D}s^{\alpha}t^{\beta}$. We distribute points in $T$ to each cell $\mathcal{C}_{i}$ containing it, then $T=\sqcup_{i=1}^{M} T_{i}$, write $t_{i}=|T_{i}|$. And we assign each cell $\mathcal{C}_{i}$ with those algebraic subsets in $S$ cutting it (no empty intersection but not containing); we call this set of algebraic subsets $S_{i}$, and $s_{i}:=|S_{i}|\leq B s \log r/r$.  

Choose $r$ large enough such that $B \log r /r$ is smaller than $\tfrac{1}{2}$, then have $s_{i}<\tfrac{1}{2}$. Now we can use induction on $s$(we already discussed the case where $s$ is very small regarding $r$). Suppose that for any collection of algebraic subsets $S_{0}$ and any finite collection of points $T_{0}\subset Y$, $t_{0}=|T_{0}|$,  such that $s_{0}=|S_{0}| < \tfrac{1}{2}s$, our bound is valid:
$I(S_{0},T_{0})\leq C'(s_{0}^{\alpha}t_{0}^{\beta}+s_{0}+t_{0}\log(2s_{0}))$

There are two possibilities for point-subvariety intersection: 

1) the given point and given subvariety are assigned to the same cell, for which the subvariety should cut this cell. And this kind of incidences is bounded by $I(S_{i}, T_{i})$;

2) the given subvariety contains a certain whole cell $\mathcal{C}_{i}$ (remember that it is not "assigned" to $\mathcal{C}_{i}$), and by Proposition, there are at most $\mathcal{O}(s_{i}+t_{i}) \leq \mathcal{O}(s+t)$ intersections. 

Now we could sum up all incidences that happen in each cell and incidences for those subvarieties that contain some cells:
\begin{align*}
I(S,T)=&\sum_{i=1}^{M}I(S_{i}, T_{i}) +r^{D}\mathcal{O}(s+t)\\
	\leq& C' \sum_{i=1}^{M}(s_{i}^{\alpha}t_{i}^{\beta}+s_{i}+t_{i}\log(2s_{i}))+r^{D}\mathcal{O}(s+t)\\
	\leq &C'(Bs\tfrac{\log{r}}{r})^{\alpha}\sum_{i=1}^{M}t_{i}^{\beta}+C'M(Bs\tfrac{\log{r}}{r})+C't\log(2B\tfrac{s\log{r}}{r})+r^{D}\mathcal{O}(s+t)\\
	\leq &C'(B \tfrac{\log{r}}{r})^{\alpha}M^{1-\beta} s^{\alpha}t^{\beta} +C' MB\tfrac{\log{r}}{r}s+C't\log(2s)+C't\log(B\tfrac{\log{r}}{r})+r^{D}\mathcal{O}(s+t)\\
	\leq &C'(B\tfrac{\log{r}}{r})^{\alpha}r^{D(1-\beta)}s^{\alpha}t^{\beta}+C'r^{D}B\tfrac{\log{r}}{r}r^{-D}s^{\alpha}t^{\beta} +C't\log(2s)-C't\\
	\leq &C'[(B\tfrac{\log{r}}{r})^{\alpha}r^{D(1-\beta)}+B\tfrac{\log{r}}{r}]s^{\alpha}t^{\beta}+C't\log(2s)+\mathcal{O}(r^{D})s+(\mathcal{O}(r^{D})-C')t
\end{align*}
The above calculations use the assumption: $s\leq r^{-D}s^{\alpha}t{\beta}$, and$B \tfrac{\log{r}}{r}\leq \tfrac{1}{2}$; together with the result of cutting lemma: $M\leq r^{D}$ and $s_{i}\leq B s \tfrac{\log{r}}{r}$.

Finally to close the induction, it is enough to choose $r$ large enough such that $(B\tfrac{\log{r}}{r})^{\alpha}r^{D(1-\beta)}+B\tfrac{\log{r}}{r}$ is smaller than 1, which is possible because $D(1-\beta)-\alpha < 0$, and then choose $C'$ larger than $\mathcal{O}(r^{D})$.
 
\end{proof}

In this incidence theorem, we do not assume transversality. And the above case 2) is where we deal with non-transversal intersections.

\begin{remark}
Very recently, Solymosi and de Zeeuw proved sharp incidence bounds in Cartesian products, using their bounds we can get rid of the $\epsilon$ in the exponent. 
\end{remark}
\begin{remark}
The proof of the Incidence Theorem uses traditional partition technique ``semi-cylindrical decomposition". Still very recently, Jacob Fox and al. proved a similar result on incidences without $K_{s,t}$ graph using new trend ``polynomial partition".
\end{remark}

\chapter{Composition Lemma}\label{section: composition lemma}
This chapter includes a proof of Composition Lemma following the structure of Elekes Szab\'{o}'s original proof with more algebraic geometry details. Many of the following claims we are going to use are known in algebraic geometry, however, we are going to prove them here for the sake of completeness. The author owes Cao Yang a lot for providing references and proofs. And this Composition Lemma is a special case of ``Group Configuration Theory" \cite{hrushovski1986contributions}. Let's recall the statement of Composition Lemma:

\begin{lemma}{$\mathbf{(Composition}$ $\mathbf{ Lemma)}$}
Given $n$-dimensional algebraic varieties $A$, $B$ and $C$. Suppose there are two $k$-dimensional families of multi-functions, $(F_{t}: A \rightarrow B, t\in T)$ and $(G_{s} :B\rightarrow C, s\in S)$, such that the family of compositions $(G_{s}\circ F_{t}: A \rightarrow C, (t,s)\in T\times S)$ have a common component with a $k$-dimensional family. Then there is a $k$-dimensional connected algebraic group $\Gamma$ acting on a variety $V$, and multi-functions $(\alpha: A \rightarrow V), (\beta: B\rightarrow V)$ and $(\gamma: C\rightarrow V)$ such that the family $F$ is related to the standard family corresponding to the $\Gamma$-action on $V$ along $\alpha$ and $\beta$, and the family $G$ is related to it along $\beta$ and $\gamma$. Moreover, the degrees of $\alpha$, $\beta$ and $\gamma$ can be bounded in terms of the degrees of the generic members $F_{t}$ and $G_{s}$.

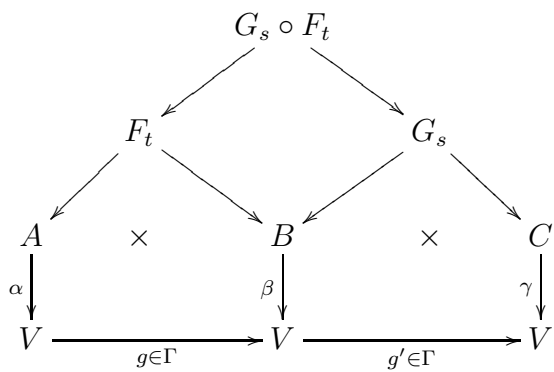
\begin{displaymath}
	\xymatrix{ &&G_{s}\circ  F_{t} \ar[dl] \ar[dr]\\
		&F_{t}\ar[dl] \ar[dr] && G_{s}\ar[dl]\ar[dr]\\
		A \ar[d]_{\alpha}&\times &B \ar[d]_{\beta}&\times &C\ar[d]_{\gamma} \\
		V\ar[rr]_{g\in\Gamma} & &V \ar[rr]_{g'\in\Gamma}&&V}
\end{displaymath}
\end{lemma}
\textbf{Notation.} Given an irreducible scheme $X$, let $\xi_{X}$ denote the generic point of $X$. For any point $x\in X$, let $k(x)$ denote the residue field of $x$. And $K(X)=k(\xi_{X})$ is the function field over $X$.

During the whole proof, we'll do the following two techniques:

1). Shrink $S$ and $T$ to smaller open sets, or to other varieties such that $F$ and $G$ are switched to equivalent families. 

2). Build a multi-function $\rho$ from $B$ (or $A$, $C$) to another variety $B'$, then replace $B$ with $B'$, and replace $F_{t}$, $G_{s}$ with a component of the composition $\rho\circ F_{t}$ and $G_{s}\circ \rho^{-1}$.

The idea of the proof is to find the ``right" irreducible component of algebraic varieties and switch $A$, $B$ and $C$ to a common algebraic variety $V$; parameter space $S$ and $T$ are included in the rational automorphism group of $V$, and the group multiplicity is map composition.

\begin{proof}
\textbf{Step one}
Throw away unneeded components, we may assume that $A$, $B$, $C$, $F$ ,$G$ $S$ and $T$ are irreducible, and we represent them with reduced scheme structure, they become integral schemes (irreducible and reduced).

\begin{claim}\label{cl: composition lemma trick}
Let $F$ be a family of multi-functions $(F_{t}: A\rightarrow B, t\in T)$ between irreducible algebraic varieties $A$ and $B$,  and is parametrized by an irreducible algebraic variety $T$. There exists a family of multi-fuctions $(F'_{t'}: A\rightarrow B, t'\in T')$ equivalent to $F$, such that every fiber $F'_{t'}$ is irreducible.
\end{claim}
\begin{proof}
We'll use the following theorem to discuss when a fiber is irreducible:

\begin{theorem}(c.f. EGA \cite{grothendieck1967elements} 4 III Theorem 9.7.7)
Let  $f : X \rightarrow S$ be a morphism of finite presentation, and let $E$ be the set of $s\in S$ such that $X_{s}$ is geometrically integral, then $E$ is locally constructible in $S$.
\end{theorem}
We consider the morphism from $F$ to $T$ induced by $F \hookrightarrow A\times B\times T\rightarrow T$. Let us first prove the Claim ~\ref{cl: composition lemma trick} under the assumption that the generic fiber $F_{\xi_{T}}$  is geometrically integral.  

Recall that a subset of $X$ is called locally closed if it is an intersection of an open subset and a closed subset. And a constructible subset of $X$ is an union of finitely many locally closed subset. As every close subset containing the generic point $\xi_{T}$ is $T$ itself, if $\xi_{T}$ is contained in $E$, then $E$ contains an open set $U$, and $U$ contains $\xi_{T}$.  After replacing $T$ with $U$, every fiber $F_{t}$ is geometrically integral, hence irreducible.

If $F_{\xi}$ is not geometrically integral, we'll switch $F$ to an equivalent family such that the new family has this property. Let $K(T)$ denote the function field of $T$, and let $L$ denote the Galois closure of the algebraic closure of $K(T)$ in $K(F)$. The field extension $L/K(T)$ is finite by Noether Normalization Lemma. Let $T'\rightarrow^{\sigma} T$ denote the normalization of $T$ regarding to $L/K(T)$, then $K(T')=L$ and we have the following diagram:

\begin{displaymath}
	\xymatrix{F\times_{T}T' \ar[r]\ar[d]^{\sigma_{F}}& T'\ar[d]^{\sigma_{T}}\\
		F \ar[r]&T}
\end{displaymath}
We apply Theorem 1.5.6 in \cite{fu2011etale} to prove that after switching $T$ to an open subset, $\sigma$ is flat.
\begin{theorem}
Let $f: X\rightarrow Y$ be a morphism of finite type between noetherian schemes and let $\mathcal{F}$ be a coherent $\mathcal{O}_{X}$-module.

(i) If $Y$ is integral, then there exists a nonempty open subset $V$ in $Y$ such that for any $x\in f^{-1}(V)$, $\mathcal{F}_{x}$ is flat over $\mathcal{O}_{Y, f(x)}$.

(ii) In general, the set $U$ of points $x\in X$ such that $\mathcal{F}_{x}$ are flat over $\mathcal{O}_{Y, f(x)}$ is open.
\end{theorem}

Without loss of generality, we assume that $\sigma_{T}: T'\rightarrow T$ is flat. As flatness is invariant under base change, $\sigma_{F}: F\times_{T}T'\rightarrow F$ is flat, a flat morphism $\sigma_{F}$ is an open map. This implies that the generic points of all irreducible components of $F\times_{T} T'$ is mapped to $\xi_{F}$. By hypothesis $F\rightarrow T$ is dominant, so the image of $\xi_{F}$ is $\xi_{T}$. Trace the commutative diagram, we observe that the map $F\times_{T}T'\rightarrow T'$ takes the generic points of all irreducible components of $F\times_{T}T'$ to $\xi_{T'}$.  Thus we have one-to-one correspondence between the following three sets: \{irreducible components of $F\times_{T}T' = \cup_{i=1}^{n}F_{i}$\},  \{irreducible components of $(F\times_{T}T')_{\xi_{T'}}=\cup_{i=1}^{n}J_{i}$\} and \{$K(F)\times_{K(T)}K(T')=\prod_{i=1}^{n}K_{i}$\}. 

We choose the irreducible component of $F\times_{T}T'$ such that if we replace $F$ with this component, $T$ with $T'$, the new family of composition $(G_{s}\circ F_{t}: A\rightarrow C, (s,t)\in S\times T)$ still has a common component with a $k$-dimensional family of multi-functions from $A$ to $C$. Let $F_{0}$ be the chosen component, and note $J_{0}=(F_{0})_{\xi_{T'}}$. Consider the morphism $F_{0}\rightarrow T'$, we have that $J_{0}$ is geometrically integral, because $K(T')$ is algebraic closed in $K(J_{0})$ by construction. And we get back to the situation where generic fiber is geometrically integral, with previous arguments, we conclude the claim. 
\end{proof}

In the first step, we reduce the problem to the special case where the compositions $G_{s}\circ F_{t}$ are irreducible for a dense set of $(s,t)$, hence the composition family itself moves in a $k$-dimensional family. We shall achieve this goal via Galois theory. 

We have first $K(B\times T)\leq K(F)$ is a finite field extension because $F\rightarrow B\times T$ is dominant and generically finite, and $\deg tr_{k}K(B\times T)=\deg tr_{k} K(F)=n+k$. Let $\tilde{F}$ denote the normalization of $F$ in the Galois closure of this field extension, and similarly, let $\tilde{G}$ denote the normalization of $G$ in the Galois closure of field extension $K(B\times S)\leq K(G)$. $\tilde{F}_{t}$ and $\tilde{G}_{s}$ will denote the fibers of the natural morphisms $\tilde{F}\rightarrow T$ and $\tilde{G}\rightarrow S$. $\tilde{F_{t}}$ and $\tilde{G_{s}}$ can still be assumed irreducible after possibly switching to an equivalent family as in the Claim ~\ref{cl: composition lemma trick}.

\begin{claim}\label{claim: galois}
There exists dense open sets $S_{0}\subset S$ and $T_{0}\subset T$, such that for every $(s,t)\in S_{0}\times T_{0}$,  the function fields $K(\tilde{F}_{t})$ and $K(\tilde{G}_{s})$ are Galois extensions of $K(B)$.
\end{claim}

\begin{proof}

Let $S_{0}=B\times T$, and let $\{S_{\lambda}\}=\text{ affine open sets of }S_{0}$. Then $S=\varinjlim_{\lambda} S_{\lambda}=Spec(K(B\times T))$. For any $S_{0}$-scheme $X_{0}$, we write $X_{\lambda}=X_{0}\times_{S_{0}}S_{\lambda}$, and $X=\varinjlim_{\lambda}X_{\lambda}=X_{0}\times_{S_{0}}S$.
 
We will apply Proposition 1.10.9 in Etale Cohomology Theory \cite{fu2011etale}. 

\begin{proposition}
Assume $S_{0}$ is quasi-compact and quasi-separated.

(i) Let $X_{0}$ and $ Y_{0}$ be $S_{0}$-schemes such that $X_{0}\rightarrow S_{0}$ is quasi-compact and quasi-separated, and that $Y_{0}\rightarrow S_{0}$ is locally of finite presentation. Then the canonical map
\[\varinjlim_{\lambda}Hom_{S_{\lambda}}(X_{\lambda}, Y_{\lambda})\rightarrow Hom_{S}(X,Y)\]
is bijective.

(ii) Suppose $X_{0}$ and $Y_{0}$ are $S_{0}$-schemes of finite presentation. If there exists an $S$-isomorphism $f : X  \xrightarrow{\cong} Y$, then for a sufficiently large $\lambda$, there exists an $S_{\lambda}$-isomorphism $f_{\lambda}:X_{\lambda}\xrightarrow{\cong}Y_{\lambda}$ including $f$.

(iii) For any $S$-scheme $X$ of finite presentation, there exists an $S_{\lambda}$-scheme $X_{\lambda}$ of finite presentation for a sufficiently large $\lambda$ such that $X\cong X_{\lambda}\times_{S_{\lambda}}S$.
\end{proposition}
\end{proof}
Write $G$ as the Galois group of field extension $K(B\times T)\leq K(\tilde{F})$, and let $\rho: G\times Spec(K(\tilde{F}))\rightarrow Spec(K(\tilde{F}))$ denote the induced morphism of Galois group action. Then we have the following induced commutative diagram on schemes:
\begin{displaymath}
	\xymatrix{G\times G\times Spec(K(\tilde{F}))\ar[r]^{id\times \rho}\ar[d]_{multiply \times id} & G\times Spec(K(\tilde{F}))\ar[d]^{\rho}\\
		G\times Spec(K(\tilde{F}))\ar[r]_{\rho} & Spec(K(\tilde{F}))}
\end{displaymath}
and
\begin{displaymath}
	\xymatrix{G\times Spec(K(\tilde{F}))\ar[dr]_{\mathbf{1}\times \pi}\ar[r]_{\rho} & Spec(K(\tilde{F}))\ar[d]_{\pi}\\
		& Spec(K(S_{0}))}
\end{displaymath}

We apply the above Proposition with $X_{0}=G\times \tilde{F}$, $Y_{0}=\tilde{F}$, as a consequence, $X=G\times Spec(K(\tilde{F}))$ and $Y=Spec(K(\tilde{F}))$. For $\lambda$ large enough, we have Galois group action on $Y_{\lambda}$:
\begin{displaymath}
	\xymatrix{G\times G\times Y_{\lambda}\ar[r]^{id\times \rho}\ar[d]_{multiply \times id} & G\times Y_{\lambda} \ar[d]^{\rho}\\
		G\times Y_{\lambda}\ar[r]_{\rho} & Y_{\lambda}}
\end{displaymath}
and
\begin{displaymath}
	\xymatrix{G\times Y_{\lambda}\ar[dr]_{\mathbf{1}\times \pi}\ar[r]_{\rho} & Y_{\lambda}\ar[d]_{\pi}\\
		& S_{\lambda}}
\end{displaymath}

More explicitly, there exists a dense open set $T_{\lambda}\subset \pi_{T}(S_{\lambda})\cap T$ ($\pi_{T}$ is the projection from $S_{\lambda}$ to $T$)  such that for every $t\in T_{\lambda}$, the group $G$ acts on $Y_{\lambda}\times_{S_{\lambda}}Spec(k(\xi_{B},t))$ by base change on fiber $(\xi_{B},t)$:

\begin{displaymath}
	\xymatrix{Y_{\lambda}\ar[r] &S_{\lambda}\\
		Y_{\lambda}\times_{S_{\lambda}}Spec(k(\xi_{B},t))\ar[r] \ar[u]&Spec(k(\xi_{B},t))\ar[u]}
\end{displaymath}

\begin{displaymath}
	\xymatrix{G\times G\times Y_{\lambda}\times_{S_{\lambda}}Spec(k(\xi_{B},t))\ar[r]^{id\times \rho}\ar[d]^{multiply\times id} & G\times Y_{\lambda}\times_{S_{\lambda}}Spec(k(\xi_{B},t))\ar[d]_{\rho}\\
	G\times Y_{\lambda}\times_{S_{\lambda}}Spec(k(\xi_{B},t))\ar[r]^{\rho} & Y_{\lambda}\times_{S_{\lambda}}Spec(k(\xi_{B},t))}
\end{displaymath}

The following commutative diagram shows that $Y_{\lambda}\times_{S_{\lambda}}Spec(k(\xi_{B},t))=\tilde{F}\times_{S_{0}}Spec(k(\xi_{B},t))$ is exactly $Spec(K(\tilde{F}_{t}))$:

\begin{displaymath}
	\xymatrix{\tilde{F}_{t} \ar[r]\ar[d]& Spec(k(t))\ar[r] & T \ar[r]&Spec(k)\\
		\tilde{F}\ar[rr] &&S_{0} \ar[u]\ar[r]&B\ar[u]\\
		\tilde{F}\times_{S_{0}}Spec(k(\xi_{B},t))\ar[u]\ar[r]&Spec(k(\xi_{B},t))\ar[rr]\ar@/^1pc/[uu]\ar[ur]&&Spec(K(B))\ar[u]\\
		&\tilde{F}_{t}\times_{B} Spec(K(B))=Spec(K(\tilde{F}_{t}))\ar@/^2pc/[uuul]\ar[urr]\ar[ul]^{\cong}}
\end{displaymath}

Rewrite group action diagram of $G$ on $Y_{\lambda}$ under base change $\times_{S_{\lambda}}Spec(k(\xi_{B},t))$ for some close point $t\in T_{\lambda}$: 
\begin{displaymath}
	\xymatrix{G\times G\times Spec(K(\tilde{F}_{t})) \ar[r]^{id\times\rho}\ar[d]^{multiply\times id}& G\times Spec(K(\tilde{F}_{t}))\ar[d]^{\rho}\\
		G\times Spec(K(\tilde{F}_{t}))\ar[r]^{\rho} & Spec(K(\tilde{F}_{t}))}
\end{displaymath}

\begin{displaymath}
	\xymatrix{G\times Spec(K(\tilde{F}_{t})) \ar[dr]_{\mathbf{1}\times \pi}\ar[r]_{\rho} & Spec(K(\tilde{F}_{t}))\ar[d]_{\pi}\\
		& Spec(k(\xi_{B},t))=Spec(K(B))}
\end{displaymath}

We can choose $Y_{\lambda}$ affine, and $G$ is a finite group acting on $Y_{\lambda}$ such that $S_{\lambda}=^{G}Y_{\lambda}$ is an small open set of $S_{0}$. We shall apply Proposition 3.1.1 in Etale Cohomology Theory \cite{fu2011etale}:

\begin{proposition}\label{prop: galois}
Let $A$ be a ring on which a finite group $G$ acts on the left, $B=A^{G}$, $X=Spec A$, $Y=Spec B$, and $\pi: X\rightarrow Y$ the morphism corresponding to the homomorphism $A^{G}\rightarrow A$.

(i)$X$ is integral over $Y$.

(ii)$\pi$ is surjective. Its fibers are orbits of $G$, and the topology on $Y$ is the quotient topology induced from $X$.

(iii) Given $x\in X$, let $y=\pi(x)$ and let
\[G_{x}=\{g\in G|gx=x\}\]
be the stabilizer of $x$. Then the residue field $k(x)$ is a normal algebraic extension of the residue field $k(y)$, and the canonical homomorphism $G_{x}\rightarrow Gal(k(x)/k(y))$ is surjective.

(iv) The canonical morphism $\mathcal{O}_{Y}\rightarrow (\pi_{*}\mathcal{O}_{X})^{G}$ is an isomorphism, and $Y$ is the quotient of $X$ by $G$.
\end{proposition}
Apply the above proposition with  $X= Y_{\lambda}$,  $Y=S_{\lambda}$ and $y=(\xi_{B},t)$. And $y$ has only one inverse image $\pi^{-1}(y)=x$,  since $Y_{\lambda}\times_{S_{\lambda}}Spec(k(\xi_{B},t))=Spec(K(\tilde{F}_{t}))$ is the spectrum of a field. By Proposition ~\ref{prop: galois} $(iii)$,  the residue field $k(x)$ is a normal algebraic extension of the residue field $k(y)$ and $G_{x} =G \rightarrow Gal(k(x)/k(y))=Gal(K(\tilde{F}_{t})/K(B))$ is surjective. We conclude that  $K(\tilde{F}_{t}) $ is Galois extension of $K(B)$ for $t$ in an dense open set of $T$, so is $K(\tilde{G}_{s})$.
\end{proof}

From the Claim ~\ref{claim: galois}, one can identify $K(\tilde{G}_{s})$ and $K(\tilde{F}_{t})$ with their unique copy in the algebraic closure of $K(B)$, moreover:

\begin{claim}
Let $L=K(\tilde{G}_{s})\cap K(\tilde{F}_{t})$, then $L$ is contained in ``nearly all" $K(\tilde{G}_{s})$ and $K(\tilde{F}_{t})$. Here ``nearly all" means for all closed point $(s,t)$ in a dense open set of $S\times T$.
\end{claim}

\begin{proof}
We shall apply the semicontinuity theorem with $X=\tilde{G}\times_{B}\tilde{F}$, $\mathcal{F}=\mathcal{O}_{X}$ and $Y$ is an open set of $T\times S$:

\begin{theorem} \cite[Theorem 12.8]{hartshorne1977algebraic}
Let $f :X\rightarrow Y$ be a projective morphism of noetherian schemes and let $\mathcal{F}$ be a coherent sheaf on $X$, flat over $Y$. Then for each $i\geq 0$, the function $h^{i}(y,\mathcal{F})=dim_{k(y)}H^{i}(X_{y},\mathcal{F}_{y})$ is an upper semicontinuous function on $Y$.
\end{theorem}

We note first that  $\dim_{k(y)}H^{0}(X_{y}, \mathcal{F}_{y})$ is finite from \cite[Theorem 9.10]{harari-cours}. Then there exists a non empty open set $W\subset Y$, such that for any $y$ in $W$, $\dim_{k(y)}H^{0}(X_{y}, \mathcal{F}_{y})$ attains the minimum. And for every close point $y\in W$, 
\begin{align*}
\dim_{k(y)} H^{0}(X_{y}, \mathcal{F}_{y})=& \# \text{ irreducible components of } X_{y}=\tilde{F}_{t}\times_{B}\tilde{G}_{s}\\
=& \# \text{ irreducible components of } Spec(K(\tilde{F}_{t})\times_{Spec(K(B))}Spec(K(\tilde{G}_{s}))\\
=& [K(\tilde{F}_{t})\cap K(\tilde{G}_{s}) : K(B)]
\end{align*}
The first equality requires the fact that $k$ is algebraically closed (see \cite[Corollaire 9.15]{harari-cours}).

From now on, we have an open set $W\subset T\times S$, such that for all closed point $y=(t,s)\in W$, the degree of field extension $[ K(\tilde{F}_{t}\cap K(\tilde{G}_{s}) :K(B)]$ is minimum. 

For a fixed $s$, $K(\tilde{G}_{s})$ is fixed. There are only finitely many Galois subfield extensions $K(B)\leq L_{i} \leq K(\tilde{G}_{s}), i=1,2...r$, such that $[L_{i}: K(B)]$ is is exactly the minimum degree of $K(B)\leq K(\tilde{F}_{t}\cap K(\tilde{G}_{s}) $.

Let $B_{L}\rightarrow B$ denote the normalization of Galois field extension $[ L: K(B)]$
\begin{displaymath}
	\xymatrix{B_{L}\ar[r]& B&\\
		\tilde{F}_{L} \ar[r]\ar[u]& \tilde{F}\ar[r]\ar[u]&T\\
		\tilde{F}_{L{t}}\ar[r]\ar[u]& \tilde{F}_{t}\ar[r]\ar[u]&Spec(k(t))\ar[u]}
\end{displaymath}
The same as previous arguments, 
\begin{align*}
H^{0}_{k(t)}(F_{L_{t}},\mathcal{O}_{\tilde{F}_{L_{t}}})&=\# \text{ irreducible components of } \tilde{F}_{L_{t}}\\
					&=\# \text{ irreducible components of } Spec(K(B_{L}))\times_{Spec(K(B))}Spec(K(\tilde{F}_{t}))\\ &= [L\cap K(\tilde{F}_{t}) :K(B)]
\end{align*}

Apply Semi-continuous Theorem again, we know that $\tilde{\Sigma}_{L} :=\{t\in T| L\cap K(\tilde{F}_{t})=L\}$ is a close set. Let $\Sigma_{L}$ denote the collection of its closed points, and write $W_{s} =\{t|(t,s)\in W\}$. $W_{s}$ is an open set of $T$, dense and irreducible. And the union of all $\tilde{\Sigma}_{L}$ is $W_{s}$, then there exists a $\tilde{\Sigma}_{i}$ contains $W_{s}$. Since $[L_{i}: K(B)]$ are minimum, all the $L_{i}$'s should be the same. 

Same arguments work for a fixed $t$, we conclude that for all closed points $(s,t)$ of $W$, $K(\tilde{F}_{t})\cap K(\tilde{G}_{s})=L$.
\end{proof}

Let's first suppose that $L=K(B)$ and then prove that the composition $G_{s}\circ F_{t}$ is irreducible. To see this, let $K(B)\leq M$ be the smallest field extension containing $K(\tilde{F}_{t})$ and $K(\tilde{G}_{s})$, and let $\tilde{B}$ be the normalization of $B$ in $M$.  Since $K(\tilde{F}_{t})\cap K(\tilde{G}_{s})=L=K(B)$, $M$ is in fact equal to $K(\tilde{F}_{t})\times_{K(B)}K(\tilde{G}_{s})$. We may apply Proposition 2. (ii) with $S_{0}=B$, $\{S_{\lambda}\}=\text{ open subsets of }S_{0}$, $X_{0}=\tilde{B}$, and $Y_{0}=\tilde{F}_{t}\times_{B} \tilde{G}_{s}$. Then $\tilde{B}=\tilde{F}_{t}\times_{B}\tilde{G}_{s}$ at least over a dense open subset of $B$. Hence $(\tilde{B}: A\rightarrow C)$ is a generalized multi-function representing the composition $G_{s}\circ F_{t}$. On the other hand $\tilde{B}$ is irreducible since it is the normalization of an irreducible scheme $B$.

In general, if $L$ is not equal to $K(B)$, let $B'$ be the normalization of $B$ in $L$, then the maps $\tilde{G}_{s}\rightarrow B$ and $\tilde{F}_{t}\rightarrow B$ will factor through $\phi : B' \rightarrow B$ for all $s$, $t$ by the universal property of normalization. Let $F_{t}'$ and $G'$ denote the image of $\tilde{F}$ and $\tilde{G}$ in $A\times B'\times T$ and $B'\times C\times S$, respectively. These are families of multi-functions $(F'_{t}: A\rightarrow B', t\in T)$ and $(G'_{s}: B'\rightarrow C, s\in S')$ such that $G_{s}=G'_{s}\circ \phi^{-1}$ and $F_{t}=\phi\circ F'_{t}$ for all $s$, $t$. (Here we treat $\phi$ as a multi-function since it is a finite morhpism and dominant.)

One can see from construction that $L=K(B')$ is the intersection of two Galois extensions $K(B)\leq K(\tilde{F}_{t})$ and $K(B)\leq K(\tilde{G}_{s})$, so $L$ is Galois over $K(B)$. So after possibly restricting to a dense open set of $B$ and the corresponding normalization,  the components of the multi-function $\phi^{-1}\circ\phi$ represented by $B'\times_{B}B'$ are just the relative automorphisms of $B'$ over $B$ (one for each element of the group $Gal(B'/B)$). But then $G_{s}\circ F_{t}=G'_{s}\circ(\phi^{-1}\circ\phi)\circ F'_{t}$ also splits according to the above Galois group. And one of these components, corresponding to a certain automorphism, moves in a $k$-parameter family. By composing each multi-function $G'_{s}$ with this automorphism we arrive to the situation that (the only) one component of $G'_{s}\circ F'_{t}$ moves in a $k$-parameter family. We shall replace $B$ with $B'$, $F$ with $F'$ and $G$ with $G'$, and we arrived at the previous situation where $K(\tilde{G}_{s})\cap K(\tilde{F}_{t})=K(B)$ for general $s$, $t$. Previous arguments imply that $G_{s}\circ F_{t}$ is irreducible, and therefore the entire family moves in a $k$-dimensional family.

\textbf{Second Step.} In this step we will study the relationship between parameter spaces of composition family and the $k$-dimensional family. Let $R$ be the $k$-dimensional subvariety of the Hilbert scheme of $A\times C$ parameterizing (almost all of) the compositions $G_{s}\circ F_{t}$. As we go along, we consider only families up to equivalence, so we shall freely shrink $R$ wo dense open subsets. By hypothesis, there is a universal family $(H_{r}: A\rightarrow C, r\in R)$ and the composition of the original families defines a rational map $\mu: T\times S\rightarrow R$ such that $G_{s}\circ F_{t}=H_{\mu(t,s)}$ for almost all pairs $(s,t)\in S\times T$.  The rational map exists because the universal property of Hilbert scheme $\mathcal{H}(A\times C)$ .After shrinking $R$ we shall assume that each $H_{r}$ is irreducible.

Next we pick a generic $r\in R$, and let $Q_{r}$ be a component of the inverse image $\mu^{-1}(r)$. 
\begin{claim}
The projections $Q_{r}\rightarrow S$ and $Q_{r}\rightarrow T$ are generically finite and surjective.
\end{claim}
\begin{proof}
$Q_{r}\subset S\times T$ is a $k$-dimensional algebraic subset and $G_{s}\circ F_{t}=H_{r}$ for almost all $(s,t)\in Q_{r}$. Then $G_{s}$ is a component of the multi-function $G_{s}\circ F_{t}\circ F_{t}^{-1}=H_{r}\circ F_{t}^{-1}$, which is independent of $s$. Hence for general $F_{t}$ there are at most finitely many $G_{s}$ with the property $G_{s}\circ F_{t}=H_{r}$ ($\approx$ the number of components of $H_{r}\circ F_{t}^{-1}$). Therefore the projection $Q_{r}\rightarrow S$ is generically finite, hence surjective since they are of the same dimension.  Then for almost all $s\in S$ there is a $(t,s)\in Q_{r}$ and similarly for almost all $t\in T$ there is a $(t,s)\in Q_{r}$.
\end{proof}

\textbf{Third Step.} This step will reduce the problem to the case where all of the multi-functions $G_{s}$ and $F_{t}^{-1}$ are graphs of rational maps $g_{s}: B\rightarrow C$ and $f_{t}: B\rightarrow A$. (i.e. They are single-value functions.)

Let $\delta$ and $\gamma$ denote the degree of the projections $H_{r}\rightarrow A$ and $G\rightarrow B\times S$. Then for general $s$ the degree of $G_{s}\rightarrow B$ is also $\gamma$. Let $C^{(\delta)}$ and $C^{(\gamma)}$ denote the symmetric powers of $C$, i.e. $C^{(\delta)}=C\times C\times...\times C /S_{\delta}$. After the claim below, we can represent the multi-functions $H_{r}$ and $G_{s}$ by rational maps $h_{r}: A\rightarrow C^{(\delta)}$ and $g_{s}: B\rightarrow C^{(\gamma)}$, i.e. for general $a\in A$, let $h_{r}(a)$ be the $\delta$-tuple of $H_{r}$-images of $a$, and $g_{s}(b)$ is defined similarly.
\begin{claim}If $H_{r}\rightarrow A$ is a multi-function of degree $\delta$ between two irreducible schemes $H_{r}$ and $A$, then there exists a canonical rational map $h_{r}: A\rightarrow C^{(\delta)}$, where $C^{(\delta)}$ is the symmetric powers of $C$.
\end{claim}
\begin{proof}
Let $\tilde{H}_{r}$ denote the normalization of $H_{r}$ in the Galois closure of $K(A)\leq K(H_{r})$. Then we can find a dense open set $A'$ of $A$ such $Gal(K(\tilde{H}_{r})/K(A))$ acts on $\tilde{H}_{r}'=\tilde{H}_{r}\times_{A}A'$ with quotient $A'$.  We will achieve the morphism $h_{r}: A\rightarrow C^{(\delta)}$ via the following diagram:

\begin{displaymath}
	\xymatrix{\tilde{H}_{r}'\times_{A'}H_{r}' \ar[rr]^{\pi_{1}}\ar[dd]^{\pi_{2}}& & H_{r}'\ar[d] &\\
			& & A'\times C\ar[d]\ar[r]&C\\
			\tilde{H}_{r}'\ar[rr]&&A'}
\end{displaymath}	
Since $\tilde{H}_{r}'\times_{A'}H_{r}'=\bigcup_{\sigma\in Gal(K(\tilde{H}_{r})/K(A))} \sigma(\tilde{H}_{r}')$, the irreducible  components of $\tilde{H}_{r}'\times_{A'}H_{r}'$ correspond to elements of Galois group $Gal(K(\tilde{H}_{r})/K(A))$, the morphism $\tilde{H}_{r}'\times_{A'}H_{r}'\rightarrow H_{r}\rightarrow A'\times C\rightarrow C$ induces $\delta$ morphisms $f_{1},...f_{\delta}$ from $\tilde{H}_{r}$ to $C$. Moreover,
\begin{displaymath}
	\xymatrix{\tilde{H}_{r}' \ar[r]\ar[d]& C\times C\times...\times C\ar[d]\\
		A' \ar[r]& C^{(\delta)}}
\end{displaymath}
For every $a\in A$, let $h_{a}$ denote one of the inverse image of $a$, then $h_{a}$ is mapped to $C^{(\delta)}$ by the above diagram, and its image is independent of choice of $h_{a}$. Hence we get  a well-defined point-wise map from $A$ to $C^{(\delta)}$. By Cor. 1.8.2, the following short sequence is exact:
\[0 \rightarrow Hom(A', C^{(\delta)})  \rightarrow Hom(H_{r}', C^{(\delta)}) \rightrightarrows^{\pi_{1}}_{\pi_{2}} Hom(H_{r}'\times_{A'}H_{r}', C^{(\delta)} ) \rightarrow 0\]
Hence we get a true morphism $h_{r}': A'\rightarrow C^{(\delta)}$, so is the rational map $h_{r}: A\rightarrow C^{(\delta)}$.
\end{proof}
Let $X\subseteq C^{(\gamma)}$ denote the set of those $\gamma$-tuples which are contained in some $\delta$-tuple in the image of $h_{r}$, it is an algebraic subset (the projection image of $h_{r}^{-1}\times_{A}F\times_{B}g_{s}$).  The image of $h_{r}$ is $n$-dimensional since $H_{r}$ is generically finite.  Every $\delta$-tuple contains only finitely many $\gamma$-tuple, because $Q_{r}\rightarrow S$ and $G_{s}\rightarrow B$ are generically finite. Hence $X$ is also $n$-dimensional. For $(t,s)\in Q_{r}$, the image of $g_{s}$ is a $n$-dimensional subvariety of $X$, thus it is a component of $X$.  And this must hold for almost all $g_{s}$. Since $g_{s}$ moves in an irreducible family, the images of $g_{s}$ must all be the same component $C^{*}\subseteq X$. Then we can replace $C$ with $C^{*}$ after possibly shrinking $S$ to a dense open subset; and the multi-functions $G_{s}$ become the graphs of the rational maps $g_{s}:B\rightarrow C^{*}$. Next we can turn around and repeat the whole argument for the compositions $F_{t}^{-1}\circ G_{s}^{-1}\supseteq H_{r}^{-1}$, and replace $A$ with some $A^{*}$, and $F_{t}^{-1}$ with rational maps $f_{t}: B\rightarrow A^{*}$. From now on, we assume that $G_{s}$ and $F_{t}^{-1}$ are graphs of the rational maps $g_{s}$ and $f_{t}$ for all $s$, $t$.

\textbf{Fourth Step.} This step corresponds to Proposition 9 of Elekes and Ronyai's paper \cite{elekes2000combinatorial}. We can look at finite (branched) covers $A'\rightarrow A$ such that each $f_{t}$ factors through it. Each $f_{t}$ has only finitely many factorizations because there are finite many subfield extensions of $K(A)\leq K(B)$, hence we can find a maximal $A'$. We replace $A$ with this cover, so from now on each such cover $A'\rightarrow A$ is an isomorphism. Similarly we can assume that $C$ has no nontrivial finite cover which is a factor of each $g_{s}$.

\textbf{Fifth Step.}
Next we reduce the problem to the case when all $f_{t}$ and $g_{s}$ are birational. For all $(t,s)\in T\times S$ the composite multi-function $G_{s}\circ F_{t}$ is just the closure of the image of the function $(f_{t}, g_{s}): B\rightarrow H_{\mu(t,s)}\subset A\times C$. Then each $f_{t}$ factors through each $H_{r}\rightarrow A$ hence $K(H_{r})=K(A)$ and $H_{r}\rightarrow A$ is birational, and so is $H_{r}\rightarrow C$. Hence each $H_{r}$ is the graph of a birational map $A\rightarrow C$. Now we fix origins $0\in S$ and $0\in T$, and replace $A$ and $C$ with $H_{\mu{(0,0)}}$ (using the projection maps). Then $H_{\mu(0,0)}$  becomes the identity multi-function. The map $(f_{0}(x), g_{0}(x))\rightarrow (f_{0}(x),g_{s}(x))$ is birational, hence the family $\{g_{s}\}$ is obtained from $g_{0}$ via composition with a $k$-dimensional family of birational automorphisms $\Gamma_{s}: C\rightarrow C$. Similarly $f_{t}=\Phi_{t}\circ f_{0}$ for another $k$-dimensional family of birational automorphisms $\Phi_{t}: A\rightarrow A$. But then $G_{s}\circ F_{t}=\Gamma_{s}\circ g_{0}\circ f_{0}^{-1}\circ \Phi_{t}^{-1} \supset \Gamma_{s}\circ \Phi_{t}^{-1}$, and these are in fact equal because of the irreducibility. This implies that we can replace $B$ with $H_{\mu(0,0)}$ using the map $(f_{0}, g_{0})$. So from now on we can assume that $A=B=C$, and $F_{t}$, $G_{s}$ are graphs of the birational automorphisms $\Phi_{t}^{-1}$ and $\Gamma_{s}$.

\textbf{Sixth Step.}
Now we replace the parameter spaces $T$, $S$ and $R$ with their images in the Hilbert scheme $\mathcal{H}(A\times A)$, so we can compare them. Let $\Psi_{r}$ denote the automorphism whose graph is $H_{r}$, then $\Gamma_{s}\circ \Phi_{t}^{-1}=\Psi_{\mu(t,s)}$. Since $\Phi_{0}^{-1}\circ \Gamma_{s}=\Gamma_{s}$, it must belong to the $\Psi$ family, and we find that $S\subset R$. But they are both irreducible and $k$-dimensional hence they are equal. Similarly $T$ is also equal to them. But then $\mu$ is an associative operation on $S$. For fixed $t$, the compositions $\Gamma_{s}\circ\Phi_{t}^{-1}$ are all different, and form a $k$-dimensional irreducible family contained in $\Psi$. Hence $\mu(0,-)$ is generically one-to-one and onto. So $\mu$ has an inverse operation on the right hand side and similarly on the left hand side as well. Therefore it is a rational group structure in the sense that both multiplication and inverse operations are rational map instead of  regular morphism. The family $\Gamma_{s}$ defines a rational action of this group on $A$. It is proved in [Weil] that up to birational equivalence we have a standard family now. Moreover, Emmanuel Breuillard proved that if $\eta>0$ is small enough, the group $\Gamma$ should be nilpotent.

\bibliography{memoireref}

\def\cprime{$'$}
\begin{thebibliography}{10}

\bibitem{chazelle1991singly}
Bernard Chazelle, Herbert Edelsbrunner, Leonidas~J Guibas, and Micha Sharir.
\newblock A singly exponential stratification scheme for real semi-algebraic
  varieties and its applications.
\newblock {\em Theoretical Computer Science}, 84(1):77--105, 1991.

\bibitem{elekes1998linear}
Gy{\"o}rgy Elekes.
\newblock On linear combinatorics ii. structure theorems via additive number
  theory.
\newblock {\em Combinatorica}, 18(1):13--25, 1998.

\bibitem{elekes1999sums}
Gy{\"o}rgy Elekes.
\newblock Sums versus products in number theory, algebra and erdos geometry.
\newblock {\em Paul Erdos and his Mathematics, II}, pages 241--290, 1999.

\bibitem{elekes2000combinatorial}
Gy{\"o}rgy Elekes and Lajos R{\'o}nyai.
\newblock A combinatorial problem on polynomials and rational functions.
\newblock {\em Journal of Combinatorial Theory, Series A}, 89(1):1--20, 2000.

\bibitem{elekes-szabo}
Gy{\"o}rgy Elekes and Endre Szab{\'o}.
\newblock How to find groups?(and how to use them in erd{\H{o}}s geometry?).
\newblock {\em Combinatorica}, 32(5):537--571, 2012.

\bibitem{erdHos1983sums}
P~Erd{\H{o}}s and Endre Szemer{\'e}di.
\newblock On sums and products of integers.
\newblock In {\em Studies in pure mathematics}, pages 213--218. Springer, 1983.

\bibitem{freaeiman1973foundations}
GA~Fre{\ae}iman.
\newblock {\em Foundations of a structural theory of set addition}, volume~37.
\newblock American Mathematical Society (Providence, RI), 1973.

\bibitem{fu2011etale}
Lei Fu.
\newblock {\em Etale cohomology theory}, volume~13.
\newblock World Scientific, 2011.

\bibitem{grothendieck1967elements}
A~Grothendieck and J~Dieudonn{\'e}.
\newblock {\'E}l{\'e}ments de g{\'e}om{\'e}trie alg{\'e}brique.
\newblock {\em Publ. math. IHES}, 8(24):2--7, 1967.

\bibitem{guth2010erdos}
Larry Guth and Nets~Hawk Katz.
\newblock On the erdos distinct distance problem in the plane.
\newblock {\em arXiv preprint arXiv:1011.4105}, 2010.

\bibitem{harari-cours}
David Harari.
\newblock {\em G{\'e}om{\'e}trie alg{\'e}brique}.
\newblock http://www.math.u-psud.fr/~harari/enseignement/geoalg/cours.pdf,
  2009/2010.

\bibitem{hartshorne1977algebraic}
Robin Hartshorne.
\newblock {\em Algebraic geometry}.
\newblock Number~52. Springer, 1977.

\bibitem{hrushovski1986contributions}
Ehud Hrushovski.
\newblock {\em Contributions to stable model theory}.
\newblock PhD thesis, University of California, Berkeley, 1986.

\bibitem{pach2013distinct}
J{\'a}nos Pach and Frank de~Zeeuw.
\newblock Distinct distances on algebraic curves in the plane.
\newblock {\em arXiv preprint arXiv:1308.0177}, 2013.

\bibitem{pach1998number}
J{\'a}nos Pach, Micha Sharir, et~al.
\newblock On the number of incidences between points and curves.
\newblock {\em Combinatorics, Probability \&amp; Computing}, 7(1):121--127,
  1998.

\bibitem{raz2014triple}
Orit~E Raz, Micha Sharir, and J{\'o}zsef Solymosi.
\newblock On triple intersections of three families of unit circles.
\newblock In {\em Annual Symposium on Computational Geometry}, page 198. ACM,
  2014.

\bibitem{raz2014polynomials}
Orit~E Raz, Micha Sharir, and J{\'o}zsef Solymosi.
\newblock Polynomials vanishing on grids: The elekes-r$\backslash$'onyai
  problem revisited.
\newblock {\em arXiv preprint arXiv:1401.7419}, 2014.

\bibitem{ruzsa1992arithmetical}
Imre~Z Ruzsa.
\newblock Arithmetical progressions and the number of sums.
\newblock {\em Periodica Mathematica Hungarica}, 25(1):105--111, 1992.

\bibitem{ruzsa1994generalized}
Imre~Z Ruzsa.
\newblock Generalized arithmetical progressions and sumsets.
\newblock {\em Acta Mathematica Hungarica}, 65(4):379--388, 1994.

\bibitem{sharir2013distinct}
Micha Sharir, Adam Sheffer, and J{\'o}zsef Solymosi.
\newblock Distinct distances on two lines.
\newblock {\em Journal of Combinatorial Theory, Series A}, 120(7):1732--1736,
  2013.

\bibitem{solymosi2009bounding}
J{\'o}zsef Solymosi.
\newblock Bounding multiplicative energy by the sumset.
\newblock {\em Advances in mathematics}, 222(2):402--408, 2009.

\bibitem{solymosi2012incidence}
J{\'o}zsef Solymosi and Terence Tao.
\newblock An incidence theorem in higher dimensions.
\newblock {\em Discrete \&amp; Computational Geometry}, 48(2):255--280, 2012.

\bibitem{spencer1984unit}
Joel Spencer, Endre Szemer{\'e}di, and William~T Trotter.
\newblock Unit distances in the euclidean plane.
\newblock {\em Graph theory and combinatorics}, pages 293--303, 1984.

\bibitem{szemeredi1983extremal}
Endre Szemer{\'e}di and William~T Trotter~Jr.
\newblock Extremal problems in discrete geometry.
\newblock {\em Combinatorica}, 3(3-4):381--392, 1983.

\bibitem{tao2012expanding}
Terence Tao.
\newblock Expanding polynomials over finite fields of large characteristic, and
  a regularity lemma for definable sets.
\newblock {\em arXiv preprint arXiv:1211.2894}, 2012.

\bibitem{wang2013bounds}
Hong Wang, Ben Yang, and Ruixiang Zhang.
\newblock Bounds of incidences between points and algebraic curves.
\newblock {\em arXiv preprint arXiv:1308.0861}, 2013.

\end{thebibliography}
\bibliographystyle{plain}

\end{document}